\documentclass[11pt]{article}
\usepackage[utf8]{inputenc} 
\usepackage[T1]{fontenc} 
\usepackage{amsmath}
\usepackage{graphicx}
\usepackage{amsfonts}
\usepackage{amssymb}
\usepackage{xcolor}
\usepackage{amsthm}
\usepackage{amscd}
\usepackage{mathrsfs}
\usepackage{float}
\usepackage{appendix}
\usepackage{subfigure}
\usepackage{glossaries}
\usepackage{stmaryrd}
\usepackage{makeidx}
\usepackage{geometry}
\usepackage{enumerate}
\usepackage[pagewise]{lineno}

\newtheorem{teo}{Theorem}[section]

\newtheorem{prop}[teo]{Proposition}
\newtheorem{lemma}[teo]{Lemma}
\newtheorem{lem}[teo]{Lemma}
\newtheorem{cor}[teo]{Corollary}
\newtheorem{defi}[teo]{Definition}
\newtheorem{oss}[teo]{Remark}

\renewcommand{\Phi}{\varPhi}
\newcommand{\R}{{\mathbb{R}}}

\newcommand{\C}{{\mathbb{C}}}

\newcommand{\N}{{\mathbb{N}}}

\def\transpose{{}^t\!}

\makeglossaries
\makeindex

\title{Semi-conical eigenvalue intersections and the ensemble controllability problem for quantum systems}

\author{Nicolas Augier\footnote{CMAP, \'Ecole Polytechnique, Institut Polytechnique de Paris}\,, 
Ugo Boscain\footnote{CNRS}\;$^\S$\,, Mario Sigalotti \footnote{Inria Paris} \footnote{Sorbonne Universit\'e, Inria, Universit\'{e} de Paris, CNRS, Laboratoire Jacques-Louis Lions, Paris, France}}

\begin{document}

\maketitle

\begin{abstract}
We study one-parametric perturbations of finite dimensional real Hamiltonians depending on two controls, and we show that generically in the space of Hamiltonians, conical intersections of eigenvalues can degenerate into \emph{semi-conical} intersections of eigenvalues.
Then, through the use of normal forms, we study the problem of ensemble controllability between the eigenstates of a generic Hamiltonian.
\end{abstract}

\section{Introduction}

Controlling parametrized families of quantum systems with a common control signal is a critical task for many applications in quantum control (see \cite{Glaser2015} and references therein), notably in Nuclear Magnetic Resonance \cite{Glaser1998}. 

For a general closed quantum system under the action of a control $u
$ and depending on a parameter $z$, the corresponding controlled equation is of the form
\begin{equation}\label{sch1}
i \frac{d\psi}{dt}(t)=H(u(t),z) \psi(t),\qquad \psi(t)\in \mathcal{H},
\end{equation}
with $H(u,z)$ self-adjoint on the complex Hilbert space $\mathcal{H}$ for every value of $u$ and $z$.   

The parameter  $z$ can be used either to describe a family of physical systems on which acts a common field driven by $u$
or a physical systems for which the value of one parameter is not known precisely.

The controllability properties of systems of this form has been studied 
both for discrete and continuous sets of parameters. 
The case of a finite set of systems is characterized in \cite{Dirr}. 
In \cite{ChittaroGauthier} the asymptotic ensemble stabization is studied for countable sets of parameters. 
In \cite{Li}, \cite{Beauchard} a proof of a strong notion of ensemble controllability 
has been obtained for a two-level system. 
Numerical ensemble control in the case of a continuum of parameters has been throughly studied for two-level systems 
\cite{Sugny,expe,KhanejaGlaser}.

Adiabatic control is a
 powerful technique which can be used to handle perturbations and uncertainties. One of its main advantages is that it provides 
 explicit and regular control laws. 
 It has hence been successfully applied to obtain 
 control strategies such as the chirp pulses (see, for instance \cite{Chirp1,Chirp2}) for spin 1/2 systems with dispersed Larmor frequency. Another nowadays classical application of adiabatic control to ensemble controllability are the so-called counterintuitive pulses for the STIRAP process \cite{Stirap1,Stirap2}. 
 A generalization of this approach has been proposed in \cite{RouchonSarlette}.
 These techniques use in an explicit or implicit way the existence of 
conical intersections between the eigenvalues of the Hamiltonian corresponding to each of the systems of the ensemble. 
A conical intersection (also called diabolic point) is a cone-like singularity of
 the spectrum of $H(u,z)$, seen as a function of the control $u$  (see Figure~\ref{intc}). 
 They are generic in the sense that they are the least degenerate 
 singularities of the spectrum of a Hamiltonian and have been studied since the beginning of quantum mechanics \cite{VonNeumann}. They 
 play an important role in the context of semiclassical analysis 
 \cite{CdV2002,CdV2004}.
Adiabatic paths through conical intersections  can be used to induce superpositions of eigenstates, as shown in \cite{Bos} and to obtain tests for exact controllability when $\mathcal{H}$ is finite-dimensional and for approximate controllability when $\mathcal{H}$ is infinite-dimensional \cite{boscain_approximate_2015}.
Results for ensemble control beyond the quantum control setting can be found in 
\cite{AgrachevBary,Helmke2,LiQi,Helmke1}.

In our paper \cite{Ensemble} we proposed a framework for the adiabatic ensemble  control of a continuum of $n$-level systems 
with real Hamiltonian, driven by two controls and having conical intersections between the eigenvalues. The main idea was that, if a system corresponding to a fixed parameter has conical intersections between two eigenvalues, then a small perturbation of the parameter yields a curve of conical intersections, each point of the curve corresponding to exactly one value of the parameter. One can then follow adiabatically such curves in the space of controls and obtain a population transfer between the two levels for the whole  
ensemble of systems. 

The argument sketched above works under the assumption that 
for all values of the parameter 
eigenvalue intersections remain conical and follow a smooth curve. 
These assumptions are satisfied for 
generic small parametric perturbations. 
For generic large perturbations it may happen that conicity of eigenvalue intersections is lost at isolated points of the curve.
The goal of this paper is to extend the analysis to this case. In particular, we 
\begin{itemize}
\item characterize typical non-conical intersections and give normal forms for them;
\item study the evolution of the system corresponding to adiabatic paths in the space of controls passing through such intersections;
\item conclude on the ensemble controllability of generic one-parameter systems 
presenting typical intersections (conical and non-conical). 
\end{itemize}

The results of this paper concern $n$-level systems with a real Hamiltonian and they can be generalized to systems evolving in an infinite-dimensional Hilbert space.
As explained in Section~\ref{NLEVEL}, thanks to the adiabatic decoupling, their study can be reduced to the case of zero-trace two-level systems. 
Consider then an equation of the form
\begin{equation}\label{ssch}
i \frac{d\psi}{dt}(t)=H(u(t),v(t)) \psi(t),\qquad \psi(t)\in \C^2,\quad (u(t),v(t))\in\R^2,
\end{equation}
with
\[
H(u,v)=\begin{pmatrix}
f_1(u,v)&f_2(u,v)\\
f_2(u,v)&-f_1(u,v)
\end{pmatrix},\qquad f=(f_1,f_2)\in  C^{\infty}(\R^2,\R^2).
\]
As in \cite{Ensemble}, we restrict our attention to real Hamiltonians, which are relevant in many physical systems, for instance for Galerkin approximations of the Schr\"odinger equation $i\partial_t \psi(x,t) =(-\Delta+V(x)+u(t)W(x))\psi(x,t)$, where $x$ belongs to a bounded set of $\R^n$ and $V$,$W$ are regular enough real functions.
The spectrum  of $H(u,v)$ is $\{\pm \sqrt{f_1(u,v)^2+f_2(u,v)^2}\}$ and, in particular, it is degenerate if and only if $f(u,v)=(0,0)$, that is, if $(u,v)$ is an \emph{eigenvalue intersection}. 
Denote by $\lambda^{+}(u,v)=\sqrt{f_1(u,v)^2+f_2(u,v)^2}$ the largest eigenvalue of $H(u,v)$ and notice that the gap (denoted $\text{Gap}(u,v)$) between the two eigenvalues of 
$H(u,v)$ is equal to $2 \lambda^{+}(u,v)$.
An eigenvalue intersection $(u,v)$ 
 is said to be \emph{conical} if \[\chi(f):=\det(\nabla f_1,\nabla f_2)\] 
is nonzero at $(u,v)$, 
where $\nabla$ denotes the gradient with respect to the variables $u$ and $v$, 
and \emph{semi-conical} 
if $\nabla f_1(u,v)$ and $\nabla f_2 (u,v)$ are collinear, are not both zero, and the directional derivative 
$\partial_{\eta}\chi(f)$ along $\eta=(-\partial_2 f_j(u,v),\partial_1 f_j(u,v))$ is nonzero if $j\in \{1,2\}$ is such that $\eta\neq 0$.
The direction spanned by $\eta$ is called the \emph{non-conical direction} at $(u,v)$.  
If $(u,v)$ is a conical intersection, then 
\begin{equation}\label{linear-growth}
\frac1C\|(u',v')-(u,v)\|\le \text{Gap}(u',v')
\le C \|(u',v')-(u,v)\|
\end{equation}
 for some $C>0$ and for all $(u',v')$ in a neighborhood of $(u,v)$. 
If, instead, 
$(u,v)$ is a semi-conical intersection, then an inequality of the type \eqref{linear-growth}
holds along any line passing through $(u,v)$ in a direction transversal to the non-conical direction.
Along the non-conical direction $\eta$
we have
\begin{equation*}
\frac1C t^2\le \text{Gap}((u,v)+t\eta)
\le C t^2
\end{equation*}
for some $C>0$ and for all $t$ in a neighborhood of $0$.

\begin{figure}[h!]\label{intnc}
  \begin{center}  \begin{minipage}[t]{0.4\linewidth}
        \includegraphics[scale=0.4]{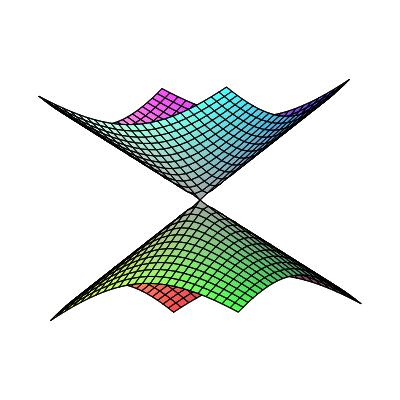}   \label{intc}
        \caption{\label{f:c}Conical intersection as a function of the controls $(u,v)\in \R^2$. }
    \end{minipage}        
\hspace{1cm}
    \begin{minipage}[t]{0.4\linewidth}
        \includegraphics[scale=0.4]{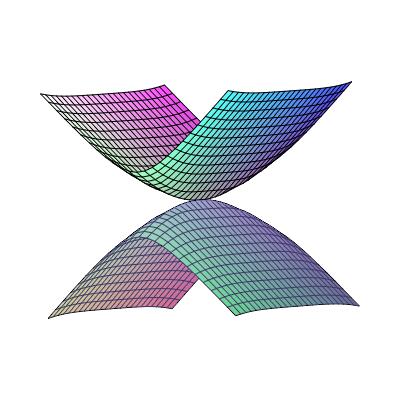}
        \caption{\label{f:nc}Semi-conical intersection of eigenvalues as a function of the controls $(u,v)\in \R^2$.}
    \end{minipage}
    \end{center}
\end{figure}

Intersections with the previous properties appear for instance in STIRAP processes with two succesive states having the same energy, that is, when, for $(u,v)\in \R^2$, $H(u,v)=\begin{pmatrix}
E&u&0\\
u&E&v\\
0&v&E'
\end{pmatrix}$, where $E<E'$.
On Figure~\ref{sc:stirap}, we have plotted the spectrum of $H(u,v)$ as a function of $(u,v)$ for such a Hamiltonian. We notice graphically that there is a semi-conical intersection between the first and second levels, and two conical intersections between the second and third levels.

\begin{figure}
\center
\includegraphics[scale=0.4]{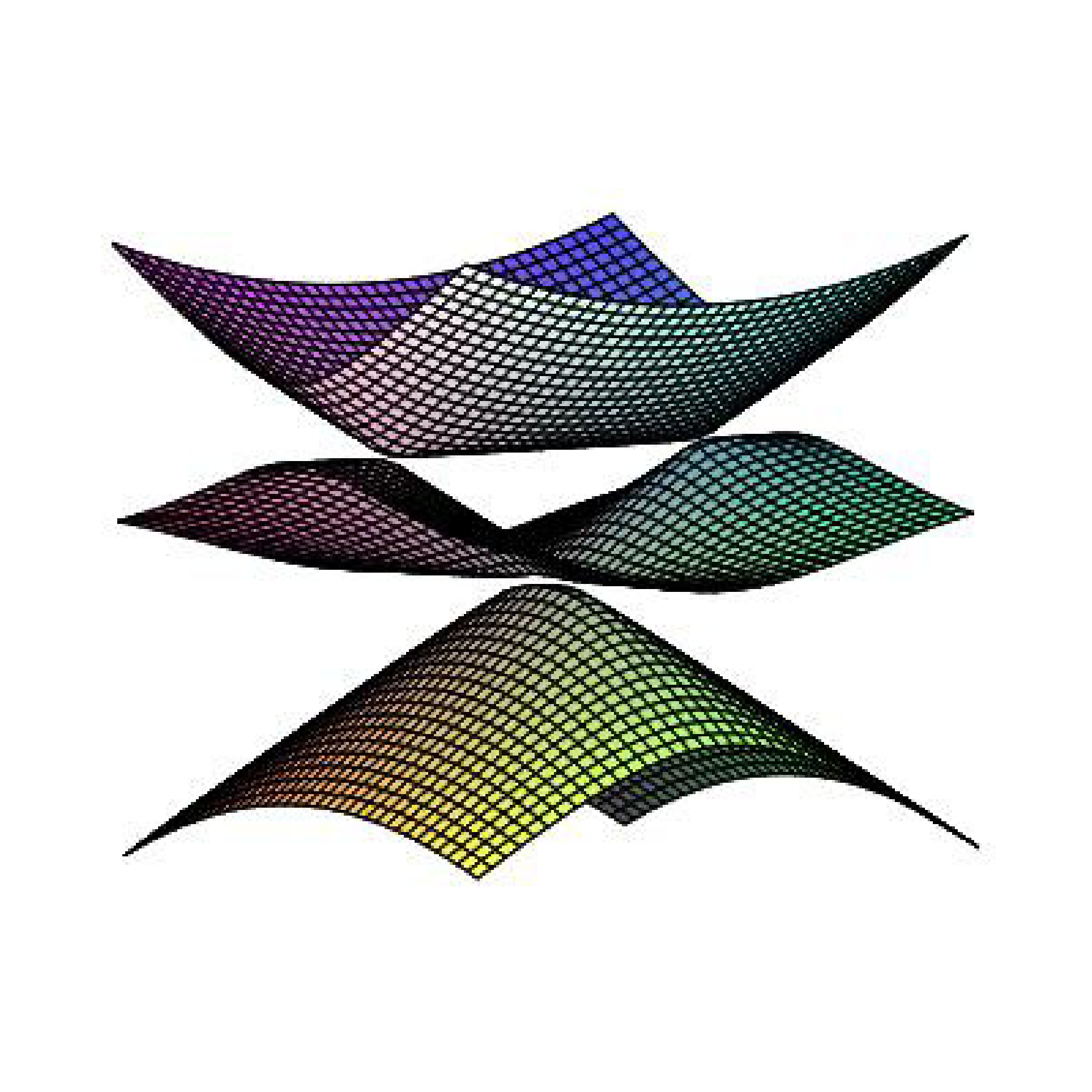}
\caption{Semi-conical intersection for the STIRAP as a function of the controls  $(u,v)\in \R^2$.}
\label{sc:stirap}
\end{figure}

Let us now consider a one-parameter family of two-level systems as above, that is,
 the \emph{Ensemble Schr\"odinger Equation} \begin{equation}\label{ssch1}
i \frac{d\psi}{dt}(t)=H(u(t),v(t),z) \psi(t),\qquad \psi(t)\in \C^2,\quad (u(t),v(t))\in\R^2,
\end{equation}
with
\[
H(u,v,z)=\begin{pmatrix}
f_1(u,v,z)&f_2(u,v,z)\\
f_2(u,v,z)&-f_1(u,v,z)
\end{pmatrix},\qquad f=(f_1,f_2)\in  C^{\infty}(\R^3,\R^2).\]
The spectrum  of $H(u,v,z)$ is $\{\lambda^{\pm}(u,v,z)=\pm \sqrt{f_1(u,v,z)^2+f_2(u,v,z)^2}\}$ and, in particular, it is degenerate if and only if $f(u,v,z)=(0,0)$.
In order to extend the definition of conical and semi-conical intersections for a one-parameter Hamiltonian, we need to add to the previous definitions some regularity assumptions with respect to the perturbation parameter $z$.
Let $(u,v,z)$ be a point such that $f(u,v,z)=(0,0)$. It is said to be conical for the family (\emph{F-conical}) if $(u,v)$ is conical for $f(\cdot,\cdot,z)$ and $\partial_3 f(u,v,z)\neq (0,0)$. It is said to be semi-conical for the family (\emph{F-semi-conical}) if it is semi-conical for $f(\cdot,\cdot,z)$ and  $f$ is a submersion at $(u,v,z)$. 
The requirement that $f$ is a submersion guarantees that 
the set
\[ Z(f)=\{(u,v,z)\mid f(u,v,z)=(0,0)\}\]
is a smooth curve in the neighborhood of a semi-conical point. 
In the following we denote by $Z_{\rm nc}(f)$ the set of non-conical 
intersections in $Z(f)$. The following lemma is a consequence of the results in Section~\ref{curve}.

\begin{lemma}\label{lem:intro}
Let $(\bar{u},\bar{v},\bar{z})
$ be a F-semi-conical intersection. Then 
$(\bar{u},\bar{v},\bar{z})$ is isolated in $Z_{\rm nc}(f)$, 
$Z(f)$ is a smooth curve locally near $(\bar{u},\bar{v},\bar{z})$, whose tangent is not vertical.
Moreover, 
the tangent at $(\bar{u},\bar{v})$
 to the projection on the plane $(u,v)$ of such a curve
 coincides with the 
non-conical direction corresponding to $(\bar{u},\bar{v},\bar{z})$.
\end{lemma}

We focus in what follows on generic properties for systems of the type \eqref{ssch1}. This means that we look for properties which hold for all $f$ in 
a subset of $C^{\infty}(\R^3,\R^2)$ with ``small'' complement. For a precise definition,
we refer to Section~\ref{Gencond}.

\begin{teo}\label{thm:intro1}
Generically with respect to $f\in C^{\infty}(\R^3,\R^2)$, for any connected component $\gamma$ 
of $Z(f)$,
\begin{enumerate}[(i)]
\item \label{p:i}
$\gamma$ is a one-dimensional submanifold of $\R^3$;
\item \label{p:ii}The  projection $\pi(\gamma)$ of $\gamma$ on the plane $(u,v)$ is a $C^{\infty}$ embedded curve
 of $\R^2$;
\item \label{p:iii}$\left(Z(f)\setminus Z_{\rm nc}(f)\right) \cap \gamma$ is made of F-conical intersections  and $Z_{\rm nc}(f)\cap \gamma$ is made of F-semi-conical intersections only.
\end{enumerate}
\end{teo}

The following theorem resumes the main properties of the control strategy that we study in the paper.
\begin{teo}\label{thm:intro2}
Assume that 
$Z(f)$ has a single connected component $\gamma$. 
 Assume moreover that
\begin{enumerate}
\item $\gamma$ satisfies properties \eqref{p:i}--\eqref{p:iii} of Theorem~\ref{thm:intro1}; 
\item $\pi(\gamma)$ has no self-intersections.
\end{enumerate}
Take two  conical intersections  $(u_0,v_0,z_0),(u_1,v_1,z_1)$ in $\gamma$ with $z_0<z_1$. 
Consider a regular $C^4$ path $(u,v):[0,1]\to \R^2$ such that 
$(u,v)(t_0)=(u_0,v_0)$, $(u,v)(t_1)=(u_1,v_1)$ for some $0<t_0<t_1<1$. 
Assume, moreover, that $(u,v)(0)=(u,v)(1)=:(\bar u,\bar v)$, that 
$(u,v)(t)\in \pi(\gamma)$ if and only if $t\in[t_0,t_1]$,  that $z\in [z_0,z_1]$ for every $z$ and $t$ such that $(u(t),v(t),z)\in \gamma
$. 
For every $z\in [z_0,z_1]$, let $\phi_{-}^z$ and $\phi_{+}^z$ be two normalized eigenvectors of $H(\bar u,\bar v,z)$ corresponding to $\lambda^{-}(\bar u,\bar v,z)$ and $\lambda^{+}(\bar u,\bar v,z)$, respectively. 
Then 
there exists $C>0$ such that for every $z\in [z_0,z_1]$
 and every $\epsilon>0$, the solution 
$\psi$ of $i \dot \psi(t)=H_f(u(\epsilon t),v(\epsilon t),z)\psi(t)$, $\psi(0)=\phi_{-}^z$, satisfies  
\[\left\|\psi\left(\frac1\epsilon\right)-e^{i\xi} \phi_{+}^z\right\|<C\epsilon^{\frac13},\]
for some $\xi\in\R$, possibly depending on $\epsilon$ and $z$.
\end{teo}

\begin{figure}[h!]
\begin{center}
       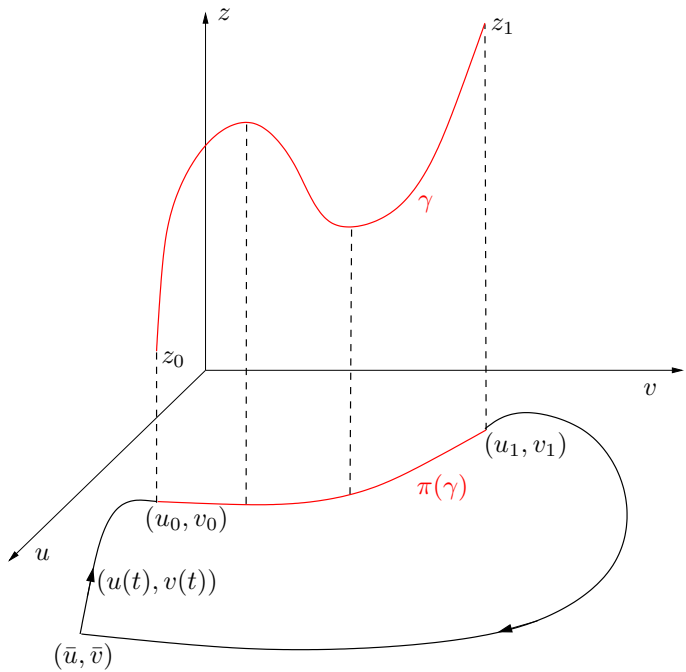
       \caption{\label{f:3D-intro}A curve $(u,v)$ as in the statement of Theorem~\ref{thm:intro2}}
\end{center}
\end{figure}

The paper is organized as follows: In Section~\ref{NORMALFORMS} we 
provide 
normal forms for the Hamiltonians yielding the  
different types of 
eigenvalue intersections 
introduced above.
In Section~\ref{curve}, we prove Theorem~\ref{thm:intro1} and we study the singularities of the projection $\pi(f)$ of $Z(f)$ on the control plane.
In Section~\ref{CONTROLSTRATEGY}, we study the dynamical properties of an isolated semi-conical intersection of eigenvalues and we prove Theorem~\ref{thm:intro2}.
Then, in Section \ref{NLEVEL}, we extend Theorem~\ref{thm:intro2} to systems with more than two levels.

\section{Basic facts and normal forms \label{normal}\label{NORMALFORMS}}
\subsection{Generic families of $2$-level Hamiltonians \label{Gencond}}

Consider a smooth function $f=(f_1,f_2):\R^2\times \R^l \to \R^2$ with $l=0$ or $l=1$.
Denote by $(e_1,\dots,e_{2+l})$ the canonical basis of $\R^{2+l}$.
Given a vector $\eta\in \R^{2+l}$ and  a smooth function $g:\R^{2+l} \to \R^q$, $q\in \N$, 
we write $\partial_\eta g$ for the directional derivative of $g$ in  the direction $\eta$ and 
$\partial_i$ for $\partial_{e_i}$, $i=1,\dots,2+l$. 
For $x\in \R^{2+l}$ and $h\in T_x \R^{2+l}\approx \R^{2+l}$, denote the differential of $f$ at $x$ applied to $h$ by $Df_x (h)$.

In the following, we study generic situations in the cases $l=0$ and $l=1$.
The coordinates $(x_1,x_2)$ play the role of controls, and are denoted by $(u,v)$, while---in the case $l=1$---the coordinate $x_3$ is a parameter and is denoted by $z$.
The space $C^\infty(\R^{2+l},\R^2)$ is endowed in what follows with the $C^{\infty}$-Whitney topology. 
We say that a property satisfied by $f\in C^{\infty}(\R^{2+l},\R^2)$ is \emph{generic} if it is satisfied in an open and dense subset of the space $C^\infty(\R^{2+l},\R^2)$.

\subsubsection{The single system case $l=0$}

Consider a two-dimensional real Hamiltonian of the form
$$H_f(u,v)=\begin{pmatrix}
f_1(u,v)&f_2(u,v) \\
f_2(u,v)&-f_1(u,v)
\end{pmatrix},$$
where $f_1,f_2:\R^2 \to\R$ are smooth functions depending on $2$ control variables $(u,v)$ and $f=(f_1,f_2)$. 
Denote by $\chi(f)(u,v)$ the Jacobian of $f$. 
Notice that the eigenvalues of $H_f$ are $\lambda^{+}=\sqrt{f_1^2+f_2^2}$ and $\lambda^{-}=-\sqrt{f_1^2+f_2^2}$.
Define $\text{Gap}=\lambda^{+}-\lambda^{-}=2\sqrt{f_1^2+f_2^2}$.
\begin{defi}
Consider $f\in C^{\infty}(\R^2,\R^2)$ and let $(\bar{u},\bar{v})\in \R^2$ be such that $f(\bar{u},\bar{v})=(0,0)$.
We say that the eigenvalue intersection $(\bar{u},\bar{v})$ is 
\begin{itemize}
\item
\emph{conical in the direction $\nu\in \R^2 
$} if $\partial_{\nu}f(\bar{u},\bar{v}) \neq (0,0)$;
\item \emph{conical} 
if $\chi(f)(\bar{u},\bar{v})\neq 0$;
\item \emph{semi-conical} 
if $\nabla f_1(\bar{u},\bar{v})$ and $\nabla f_2(\bar{u},\bar{v})$ are collinear, are not both zero, and the directional derivative of 
$
\chi(f)(\bar{u},\bar{v})$ along $\eta=(-\partial_2 f_j(\bar{u},\bar{v}),\partial_1 f_j(\bar{u},\bar{v}))$ is nonzero if $j\in \{1,2\}$ is such that $\eta\neq 0$.
The direction spanned by $\eta$ is called the \emph{non-conical direction} at $(\bar{u},\bar{v})$. 
\end{itemize}
\end{defi}

\begin{oss}

The definition of conical intersection for $f\in C^{\infty}(\R^2,\R^2)$ given above is equivalent to the one used in \cite{Bos}, namely, 
a point $(\bar{u},\bar{v})$ such that $f(\bar{u},\bar{v})=(0,0)$ and there exists $c>0$ such that, for every $\eta\in \R^2$ of norm $1$ and $\delta>0$ small enough, we have  $\frac{\text{Gap}((\bar{u},\bar{v})+\delta \eta)}{\delta}\ge c$.
\end{oss}
The following proposition states that semi-conical points are isolated zeros of $f$.
Although the proof could be deduced rather straightforwardly from the definition, we postpone it 
for simplicity 
to Section \ref{unpnmf}, where we base it on normal forms.
\begin{prop}[Semi-conical intersections are isolated]\label{sc:isolated}
Consider a \emph{semi-conical} point $(\bar{u},\bar{v})$ for $f$. 
Then there exists a neighborhood $V$ of $(\bar{u},\bar{v})$ in $\R^2$ such that for every $(u,v)\in V\setminus\{(\bar u,\bar v)\},$ $f(u,v)\neq (0,0)$.
\end{prop}

We introduce here the transversality argument used in the following of the 
section to prove genericity of several properties. As an illustration, we recall how 
such an argument can be used 
to prove that conical intersections are generic.
Denote by 
$J^{1}(\R^{2},\R^{2})$ the space of  $1$-jets of functions from $\R^{2}$ to $\R^{2}$.
For every $(u,v)\in \R^2$ and $f\in C^1(\R^2,\R^2)$, we write $j^1 (f) (u,v)\in J^1(\R^2,\R^2)$ to denote the $1$-jet of $f$ at $(u,v)$. 
Define 
\[
S_r=\{j^1 (f) (0)
\mid  f\in C^{\infty}(\R^{2},\R^{2})\;f(0)=0,\; \text{rank}(Df (0))=r \},\qquad r=0,1.\\
\]
It is easy to check that $S_0$, $S_1$ are two 
submanifolds of  $J^1(\R^2,\R^2)$ of codimension $6$ and $3$, respectively. 
One can easily show that the algebraic subset $S_0\cup S_1$ of $J^{1}(\R^{2},\R^{2})$ admits a Whitney stratification (see \cite{Goresky} Part I, Chapter $1$) whose strata have codimension strictly larger than the dimension of $\R^2$.
By Thom’s transversality theorem (see, e.g., \cite{Golub}) used in combination with \cite[\textsection{1.3.2}]{Goresky},
\begin{align*}
U&=\{f\in C^{\infty}(\R^{2},\R^{2}) \mid j^{1}(f)(\R^{2})\cap (S_1\cup S_2)=\emptyset\}\\
&=\{f\in C^{\infty}(\R^{2},\R^{2}) \mid j^{1}(f)(\R^{2})\cap S_1=\emptyset\}\cap \{f\in C^{\infty}(\R^{2},\R^{2}) \mid j^{1}(f)(\R^{2})\cap  S_2=\emptyset\}
\end{align*}
 is an open and dense subset of $C^{\infty}(\R^{2},\R^{2})$.

\subsubsection{The ensemble case $l=1$}
Consider a two-dimensional real Hamiltonian of the form
$$H_f(u,v,z)=\begin{pmatrix}
f_1(u,v,z)&f_2(u,v,z)\\
f_2(u,v,z)&-f_1(u,v,z)
\end{pmatrix},$$
where $f_1,f_2:\R^3 \to\R$ are smooth functions depending on two control variables $(u,v)$ and one parameter $z$.
Define the smooth function $f=(f_1,f_2):\R^3\to \R^2$.
An \emph{eigenvalue intersection} 
is a point $(u,v,z)$ such that $f(u,v,z)=(0,0)$. 
\begin{defi}
For $i,j\in \left\{1,2,3\right\}$, let $\chi_{ij}(f)$
 be the 
Jacobian
of the restriction of $f$ to the plane
 parallel to $\text{span}(e_i,e_j)$, i.e.,
$$\chi_{ij}(f)(u,v,z)=\begin{vmatrix}
\partial_i f_1(u,v,z)&\partial_j f_1(u,v,z) \\
\partial_i f_2(u,v,z)&\partial_j f_2(u,v,z)
\end{vmatrix}.$$
 By a slight abuse of notation, we set $\chi(f)=\chi_{12}(f)$.
\end{defi}

In order to extend the definition of conical and semi-conical intersections to parametrized Hamiltonians, we need to add to the previous definitions some regularity assumptions with respect to the parameter $z$.
\begin{defi}
Let $(u,v,z)$ be a point such that $f(u,v,z)=(0,0)$. It is said to be conical for the family (\emph{F-conical}) if 
$(u,v)$ is conical for $f(\cdot,\cdot,z)$ and $\partial_3 f(u,v,z)\neq (0,0)$. 
It is said to be semi-conical for the family (\emph{F-semi-conical}) 
if it is semi-conical for $f(\cdot,\cdot,z)$
and  $f$ is a submersion at $(u,v,z)$. 
\end{defi}

\begin{prop}\label{sub}
Generically with respect to  $f\in C^{\infty}(\R^3,\R^2)$, $f$ is a submersion at every point of $Z(f)$ and the set $Z(f)=\{(u,v,z)\in \R^{3} \mid f(u,v,z)=0 \}$ is a 
submanifold of $\R^3$ of dimension $1$. 
\end{prop}
\begin{proof}
Define
\[\Sigma_r=\{j^{1}(f)(0)\in J^{1}(\R^{3},\R^{2}) \mid f\in C^\infty(\R^3,\R^2),\;f(0)=0,\; \text{rank}(Df (0))=r\},\qquad r=0,1.
\]
Notice that $\Sigma_0$ and $\Sigma_1$ are smooth submanifolds of $J^{1}(\R^{3},\R^{2})$ of codimensions $8$ and $4$, respectively. 
One can easily show that the algebraic subset $\Sigma_0\cup \Sigma_1$ of $J^{1}(\R^{3},\R^{2})$ admits a Whitney stratification  
whose strata have codimension strictly larger than $3$. 
Transversality theory then allows to conclude
 that
$U=\{f\in C^{\infty}(\R^{3},\R^{2})\mid  j^{1}(f)(\R^{3})\cap (\Sigma_0\cup \Sigma_1)=\emptyset\}$ is a an open and dense 
subset of $C^{\infty}(\R^{3},\R^{2})$.
Hence, $f$ is generically a submersion at every point $(u,v,z)\in Z(f)$. The proposition is proved.
\end{proof}

 In the next two propositions we provide a geometric description of the curve $Z(f)$ and we show its links with the conicity properties of $f$.
\begin{prop} \label{trans}
A point $(\bar{u},\bar{v},\bar{z})\in Z(f)$ 
 is conical for $f(\cdot,\cdot,\bar{z})$ if and only if $f$ is a submersion at $(\bar{u},\bar{v},\bar{z})$ such that $Z(f)$ is locally near $(\bar{u},\bar{v},\bar{z})$ a one-dimensional submanifold transversal to the plane of $\R^3$ of equation $z=\bar{z}$. 
 \begin{proof}
Let $(\bar{u},\bar{v},\bar{z})\in Z(f)$ be conical for $f(\cdot,\cdot,\bar{z})$. By definition, we have $\chi(f)(\bar{u},\bar{v},\bar{z})\neq 0$, hence $f$ is a submersion at $(\bar{u},\bar{v},\bar{z})$. It follows that $Z(f)$ is locally near $(\bar{u},\bar{v},\bar{z})$ a one-dimensional submanifold of $\R^3$. Fix $\bar{t}\in \R$ and
a local smooth regular parametrization $c(t)=(u(t),v(t),z(t))_{t\in \R}$ of $c \subset Z(f)$ such that $c(\bar{t})=(\bar{u},\bar{v},\bar{z})$. 
Assume for the sake of contradiction that $\dot{z}(\bar{t})=0$.
Differentiating the condition $f(c(t))=0$, we have 
 \begin{equation}\label{van}
\left\{\begin{aligned}
&\dot{u}(\bar{t})\partial_1 f_1(\bar{u},\bar{v},\bar{z})+\dot{v}(\bar{t})\partial_2 f_1(\bar{u},\bar{v},\bar{z})=0\\
&\dot{u}(\bar{t})\partial_1 f_2(\bar{u},\bar{v},\bar{z})+\dot{v}(\bar{t})\partial_2 f_2(\bar{u},\bar{v},\bar{z})=0
\end{aligned}\right.
\end{equation} 
Hence $\chi(f)(\bar{u},\bar{v},\bar{z})=0$, that is impossible.

Conversely, consider a submersion $f$ at $(\bar{u},\bar{v},\bar{z})$ such that $Z(f)$ is locally near $(\bar{u},\bar{v},\bar{z})$ a one-dimensional submanifold transversal to the plane of $\R^3$ of equation $z=\bar{z}$. 
For the sake of contradiction, assume that $(\bar{u},\bar{v},\bar{z})$ is non conical for $f(\cdot,\cdot,\bar{z})$.
By definition, there exists a direction $\eta\in \R^2\setminus\{(0,0)\}$ such that $\partial_{(\eta,0)}f(\bar{u},\bar{v},\bar{z})=0$. 
Fix $\bar{t}\in \R$ and
a local smooth regular parametrization $c(t)=(u(t),v(t),z(t))_{t\in \R}$ of $c \subset Z(f)$ such that $c(\bar{t})=(\bar{u},\bar{v},\bar{z})$.
Differentiating the condition $f(c(t))=0$, we have $(\dot{u}(\bar{t}),\dot{v}(\bar{t}),\dot{z}(\bar{t})) \in \ker \text{D}f_{(\bar{u},\bar{v},\bar{z})}$. 
Since $f$ is a submersion at $(\bar{u},\bar{v},\bar{z})$, we deduce that $(\dot{u}(\bar{t}),\dot{v}(\bar{t}),\dot{z}(\bar{t}))$  is collinear to $(\eta,0)$. Hence we get $\dot{z}(\bar{t})=0$, which is impossible.
 \end{proof}
\end{prop}

\begin{prop}\label{order:tangency}
Assume that $f$ is a submersion locally near $(\bar{u},\bar{v},\bar{z})$ and that $(\bar{u},\bar{v},\bar{z})$ is non-conical for $f(\cdot,\cdot,\bar{z})$ in the direction $\eta\in \R^2 \setminus \left\{(0,0)\right\}$. Fix 
$\bar{t}\in \R$
and a local smooth regular parametrization $c(t)=(u(t),v(t),z(t))_{t\in \R}$ of $c \subset Z(f)$ such that $c(\bar{t})=(\bar{u},\bar{v},\bar{z})$ and $\dot{z}(\bar{t})=0$.
Then we have the equivalence $$\ddot{z}(\bar{t})= 0 \iff \partial_{(\eta,0)} \chi(f)(\bar{u},\bar{v},\bar{z})= 0.$$
In particular, if $(\bar{u},\bar{v},\bar{z})$ is F-semi-conical for $f$ then $\ddot{z}(\bar{t})\neq 0$.
\begin{proof}
Without loss of generality, we can assume $(\bar{u},\bar{v},\bar{z})=(0,0,0)$.
Under the assumption that $f$ is a submersion, we have $(\partial_1 f_1(0),\partial_2 f_1(0))\neq (0,0)$ or $(\partial_1 f_2(0),\partial_2 f_2(0))\neq (0,0)$.
Without loss of generality, assume  $\partial_1 f_1(0)=r \cos(\theta)$ and $\partial_2 f_1(0)=r \sin(\theta)$ where $r>0$ and $\theta\in [0,2\pi]$.
Define the matrix
$R_{\theta}=\begin{pmatrix}
-\sin(\theta) & -\cos(\theta)&0 \\
\cos(\theta)&-\sin(\theta)&0 \\
0&0&1
\end{pmatrix}$. 
For every $(u,v,z)\in \R^3$, define $\tilde{f}(u,v,z)=(f\circ R_{\theta}) (u,v,z)$.
By simple computations, we have $\partial_1 \tilde{f}(0)=0$ 
and 
$\partial_1 \chi (\tilde{f})
=\partial_{(\eta,0)} \chi (f)
$.
Hence, it is sufficient to prove the proposition for $\eta=(1,0)$.

Notice that 
\begin{equation}\label{eq:chif0}
\partial_1\chi(f)(0)=\begin{vmatrix}
\partial_{11} f_1(0)&\partial_2 f_1(0) \\
\partial_{11} f_2(0)&\partial_2 f_2(0)
\end{vmatrix}+
\begin{vmatrix}
\partial_{1} f_1(0)&\partial_{21} f_1(0) \\
\partial_{1} f_2(0)&\partial_{21} f_2(0)
\end{vmatrix}=\begin{vmatrix}
\partial_{11} f_1(0)&\partial_2 f_1(0) \\
\partial_{11} f_2(0)&\partial_2 f_2(0)
\end{vmatrix},
\end{equation}
where we used that  $\partial_1 f_1(0)=\partial_1 f_2(0)=0$.

Since $f$ is a submersion at $0$, the equality  $\left.\frac{d}{dt}f(c(t))\right|_{t=\bar t}=0$ implies that $\dot{c}(\bar{t})= ae_1$ for some $a\ne 0$. 
 The equality $\frac{d^2}{dt^2}f(c(t))=0$ can be rewritten as
$Df_{c(t)}(\ddot{c}(t))+D^{2}f_{c(t)}(\dot{c}(t),\dot{c}(t))=0$, 
that is, for $t=\bar t$,
\begin{equation*}\left\{
\begin{aligned}
\ddot{y}(\bar{t}) \partial_2 f_1(0)+\ddot{z}(\bar{t})\partial_3 f_1(0) +a^2\partial_{11} f_1(0)=0,\\
\ddot{y}(\bar{t})\partial_2 f_2(0)+\ddot{z}(\bar{t})\partial_3 f_2(0) +a^2\partial_{11} f_2(0)=0,
\end{aligned}\right.
\end{equation*}
which can be rewritten as
\begin{equation}\label{systlin}
\begin{pmatrix}
\partial_{11} f_1(0)&\partial_2 f_1(0) \\
\partial_{11} f_2(0)&\partial_2 f_2(0)
\end{pmatrix}
\begin{pmatrix}
a^2\\ 
\ddot{y}(\bar{t})
\end{pmatrix}=-\ddot{z}(\bar{t})\partial_3 f(0).
\end{equation}

Since $f$ is a submersion at $0$ and $\partial_1 f(0)=0$, we have that $\partial_3 f(0)$ is nonzero. The conclusion then follows from \eqref{eq:chif0} and \eqref{systlin}.
\end{proof}
\end{prop} 

\begin{oss}\label{rmk:isolated}
If $(\bar{u},\bar{v},\bar{z})$ is F-semi-conical and $(u(t),v(t),z(t))_{t\in \R}$ such that $(u(\bar{t}),v(\bar{t}),z(\bar{t}))=(\bar{u},\bar{v},\bar{z})$ is a smooth and regular local parametrization of $Z(f)$, then $\dot{z}(\bar{t})= 0$ and $\ddot{z}(\bar{t})\neq 0$.
In particular, F-semi-conical intersections are isolated in $\R^3$.
\end{oss}

 The following two propositions guarantee that for a generic $f$, all intersections are either F-conical or F-semi-conical.
 \begin{prop}\label{genconiq}
For a generic $f\in C^\infty(\R^{3},\R^{2})$, 
for every $(u,v,z)\in \R^3$ such that $(u,v,z)$ is a conical intersection for $f(\cdot,\cdot,z)$, we have that $(u,v,z)$ is a F-conical intersection for $f$.
\end{prop}
  \begin{proof}
The set $Q=\{j^{1}(f)(0)\in J^{1}(\R^{3},\R^{2}) \mid f(0)=0,\partial_3 f(0)=0  \}$ is a Whitney stratified subset of $J^{1}(\R^{3},\R^{2})$ of codimension $4$.
By 
transversality 
theory
the set
$\{f\in C^{\infty}(\R^{3},\R^{2})\mid  j^{1}(f)(\R^{3})\cap Q=\emptyset\}$ is an open and dense subset of $C^{\infty}(\R^{3},\R^{2})$. 
\end{proof}

\begin{prop}\label{ordre2}
For a generic $f\in C^\infty(\R^{3},\R^{2})$, 
for every $(u,v,z)\in \R^3$ such that $(u,v,z)$ is a non-conical intersection for $f(\cdot,\cdot,z)$, we have that $(u,v,z)$ is a F-semi-conical intersection for $f$.
\end{prop}
\begin{proof}
 Set
\[ S_j=\left\{j^{2}(f)(0)\in J^{2}(\R^{3},\R^{2})\mid
   f(0)=(0,0), \partial_1 f_j(0)=\partial_2 f_j(0)=0    
    \right\},\qquad j=1,2.\]
Then $S_1$ and $S_2$ are smooth subspaces of $J^{2}(\R^{3},\R^{2})$ of codimension $4$.
Define
\[\eta=(-\partial_2 f_1(0),\partial_1 f_1(0),0),\]
\[ S_3=\left\{j^{2}(f)(0)\in J^{2}(\R^{3},\R^{2})\mid f(0)=0,\;(\partial_1 f_1(0),\partial_2 f_1(0))\neq 0, \; \chi(f)(0)=0, \;  \partial_{\eta}\chi(f)(0)= 0   \right\}\] and
\[ \tilde{S}_3=\left\{j^{2}(f)(0)\in J^{2}(\R^{3},\R^{2})\mid f(0)=0,\; \chi(f)(0)=0, \;  \partial_{\eta}\chi(f)(0)= 0   \right\}.\]

We are going to prove that $S_3$  is a smooth submanifold of $J^{2}(\R^{3},\R^{2})$ of codimension $4$, that is, that the equalities $f(0)=0$, $\chi(f)(0)= 0$ and $\partial_{\eta}\chi(f)(0)= 0$ define independent equations in $J^{2}(\R^{3},\R^{2})$
under the condition $(\partial_1 f_1(0),\partial_2 f_1(0))\neq 0$.
The equality $f(0)=0$ is clearly independent from the two others. Using the property that 
$\partial_1 f_1(0)\ne 0$ or $\partial_2 f_1(0)\ne 0$, one easily establishes 
that
$\chi(f)(0)= 0$ and $\partial_{\eta}\chi(f)(0)= 0$ define independent equations in $J^{2}(\R^{3},\R^{2})$.

One then can easily prove that the algebraic subset $S=S_1\cup S_2 \cup \tilde{S}_3=S_1\cup S_2\cup S_3$ of $J^2(\R^3,\R^2)$ admits a Whitney stratification whose strata have a codimension strictly larger than 3.
By 
transversality theory
we get that
$O=\{f\in C^{\infty}(\R^{3},\R^{2})\mid  j^{2}(f)(\R^{3})\cap S=\emptyset\}$ is an open and dense subset of $C^{\infty}(\R^{3},\R^{2})$.
\end{proof}

\subsection{Admissible transformations}\label{adm}

The aim of this section is to define admissible transformations in order to get normal forms for the Hamiltonians $H_f=\begin{pmatrix}
f_1&f_2 \\
f_2&-f_1
\end{pmatrix}$ defined for $f\in C^{\infty}(\R^2,\R^2)$ 
 and  $f\in C^{\infty}(\R^3,\R^2)$.
Consider the \emph{ Schr\"odinger Equation}, defined for $f\in C^{\infty}(\R^2,\R^2)$ by 
\begin{equation}\label{schsolo}
i \frac{d\psi(t)}{dt}=H_f(u(t),v(t)) \psi(t),\qquad \psi(t)\in \C^2,
\end{equation} 
and the \emph{Ensemble Schr\"odinger Equation}, defined for $f\in C^{\infty}(\R^3,\R^2)$ by
\begin{equation}\label{sch}
i \frac{d\psi(t)}{dt}=H_f(u(t),v(t),z) \psi(t),\qquad \psi(t)\in \C^2.
\end{equation} 
The control functions $u,v$ are in $L^{\infty}(\R,\R)$ and the perturbation $z$  belongs to $[z_0,z_1]$ where $z_0,z_1\in \R $.

The three transformations
correspond
to equivalent representations of the dynamical systems~(\ref{schsolo}) and (\ref{sch})
achieved, respectively, by  time-reparameterization, state-space diffeomorphism, and 
independent diffeomorphic transformations of both the space of controls and the space of  perturbations.

\begin{defi}\label{d:time-e}
We say that two elements $f$ and $\tilde f$ of $C^{\infty}(\R^2,\R^2)$ (respectively, $C^{\infty}(\R^3,\R^2)$) are \emph{time-equivalent} at $0$ if there exists a nowhere-vanishing function $\xi\in C^{\infty}(\R^2,\R)$ such that
$\tilde{f}(u,v)=\xi(u,v)f(u,v)$  (respectively, $\tilde{f}(u,v,z)=\xi(u,v)f(u,v,z)$) in a neighborhood of $0$. 
\end{defi}
\begin{oss}
A time-equivalence as introduced in Definition~\ref{d:time-e} with $\xi>0$ corresponds to a time-change in Equation~(\ref{schsolo}). As for the case $\xi<0$, consider $f\in C^{\infty}(\R^2,\R^2)$, $\psi_0,\psi_1\in \C^2$ and a control path $(u(\cdot),v(\cdot))$ such that the solution $\psi:[0,T]\to \C^2$ of Equation~(\ref{schsolo}) with $\psi(0)=\psi_0$ satisfies $\psi(T)=\psi_1$. Then the solution $\tilde{\psi}$ of Equation~(\ref{schsolo}) associated with $(u(T-\cdot),v(T-\cdot))$ such that $\tilde{\psi}(0)=\bar{\psi}_1$ satisfies $\tilde{\psi}(T)=\bar{\psi}_0$ (where we denote by $\bar{x}$ the complex-conjugate of $x\in \C^2$). Hence the equations (\ref{schsolo}) and
\begin{equation*}
i \frac{d\psi(t)}{dt}=-H_f(u(t),v(t)) \psi(t),\qquad \psi(t)\in \C^2,
\end{equation*} have the same controllability properties.
 Hence time-equivalence is justified 
 for a 
 function $\xi\in C^{\infty}(\R^2,(-\infty,0))$. 
The same argument is also valid for $f\in C^{\infty}(\R^3,\R^2)$.
\end{oss}

\begin{defi}
We say that two elements $f$ and $\tilde f$ of $C^{\infty}(\R^3,\R^2)$ 
or
$C^{\infty}(\R^2,\R^2)$ 
are \emph{left-equivalent} if there exists $P\in {\rm O}_{2}(\R)$ independent of $u,v,z$ 
 such that $H_f=PH_{\tilde{f}}P^{-1}$.
\end{defi}

\begin{oss}\label{left}
Let $f$ be in $C^{\infty}(\R^3,\R^2)$ or $C^{\infty}(\R^2,\R^2)$.
Considering \[P_{\theta,\zeta}=\begin{pmatrix}
\cos(\theta)&-\zeta \sin(\theta)\\
\sin(\theta)&\zeta \cos(\theta)
\end{pmatrix}\in {\rm O}_2(\R),\] where $\theta\in \mathbb{S}^1$ and $\zeta=\pm 1$, the associated left-equivalence transforms $f=(f_1,f_2)$ into \[\tilde{f}=(\cos(2\theta)f_1-\zeta \sin(2\theta)f_2,\zeta \cos(2\theta)f_2+\sin(2\theta)f_1).\]
\end{oss}

\begin{oss}
Let $f$ be in $C^{\infty}(\R^3,\R^2)$ or $C^{\infty}(\R^2,\R^2)$.
If $t\mapsto \psi(t)$ is a solution of Equation~(\ref{schsolo}) or (\ref{sch})  associated with $f$ and with initial condition $\psi(0)=\psi_0$, then $t\mapsto Y(t)=P\psi(t)$ is a solution of Equation~(\ref{sch}) associated with $\tilde{f}$ satisfying $Y(0)=P\psi_0$. Hence, transitions for $Y$ between the eigenstates of $H_{\tilde{f}}=PH_{f}P^{-1}$  correspond to transitions for $\psi$ between the eigenstates of $H_{f}$.
\end{oss}

\begin{defi}We say that two elements $f$ and $\tilde f$ of $C^{\infty}(\R^2,\R^2)$ 
are \emph{right-equivalent} at $0$ if there exists a diffeomorphism $\phi\in C^{\infty}(\R^2,\R^2)$ such that $\phi(0)=0$ and  $\tilde{f}=f\circ \phi$ in a neighborhood of $0$.
\end{defi}

\begin{defi}\label{decoupling}
We say that two elements $f$ and $\tilde f$ of $C^{\infty}(\R^3,\R^2)$ 
are \emph{right-equivalent} at $0$ if there exists a diffeomorphism $\phi\in C^{\infty}(\R^3,\R^3)$ 
of the form $\phi:(u,v,z)\mapsto
(\phi_1(u,v),\phi_2(u,v),\phi_{3}(z))$, where $\phi_1,\phi_2 \in C^{\infty}(\R^2,\R)$  and $\phi_3 \in C^{\infty}(\R,\R)$, 
satisfying $\phi(0)=0$ and 
$\tilde{f}=f\circ \phi$ in a neighborhood of $0$.
\end{defi}

Combining the previous three definitions we introduce the following  notion of equivalence.
\begin{defi}\label{eq:1}
We say that two elements $f$ and $\tilde f$ of $C^{\infty}(\R^2,\R^2)$ 
are \emph{equivalent} at $0$ if there exists $(\phi,\theta,\xi)\in C^{\infty}(\R^2,\R^2)\times \mathbb{S}^1 \times C^{\infty}(\R^2,\R\setminus\{0\})$ with $\phi$
as in Definition~\ref{decoupling}, and $\zeta=\pm 1$ such that for every $(u,v,z)$ in a neighborhood of $0$,
\[\left\{\begin{aligned}
\tilde{f}_1(u,v)=\xi(u,v)(\cos(2\theta)f_1\circ \phi(u,v)-\zeta \sin(2\theta)f_2 \circ \phi(u,v)),\\
\tilde{f}_2(u,v)=\xi(u,v)(\sin(2\theta)f_1 \circ \phi(u,v)+\zeta \cos(2\theta)f_2 \circ \phi(u,v)).
\end{aligned}\right.\]
\end{defi}
\begin{defi}\label{eq}
We say that two elements $f$ and $\tilde f$ of $C^{\infty}(\R^3,\R^2)$ 
are \emph{equivalent} at $0$ if there exists $(\phi,\theta,\xi)\in C^{\infty}(\R^3,\R^{3})\times \mathbb{S}^1 \times C^{\infty}(\R^2,\R\setminus\{0\}) $ with $\phi$ 
as in Definition~\ref{decoupling}, and $\zeta=\pm 1$ such that for every $(u,v,z)$ in a neighborhood of $0$,
\[\left\{\begin{aligned}
\tilde{f}_1(u,v,z)=\xi(u,v)(\cos(2\theta)f_1\circ \phi(u,v,z)-\zeta \sin(2\theta)f_2 \circ \phi(u,v,z)),\\
\tilde{f}_2(u,v,z)=\xi(u,v)(\sin(2\theta)f_1 \circ \phi(u,v,z)+\zeta \cos(2\theta)f_2 \circ \phi(u,v,z)).
\end{aligned}\right.\]
\end{defi}

An essential feature of admissible transformations is the following proposition, which is obtained by a direct application of the definitions.
\begin{prop}\label{c:conserved}
\begin{itemize}
\item Let  $f,\tilde f\in C^{\infty}(\R^2,\R^2)$ 
be  equivalent. Then $0$ is conical for $f$ if and only if it 
is conical for $\tilde{f}$ and $0$ is semi-conical for $f$ if and only if it  
is semi-conical for $\tilde{f}$. 
\item Let  $f,\tilde f\in C^{\infty}(\R^3,\R^2)$ 
Then $0$ is F-conical for $f$ if and only if it  
is F-conical for $\tilde{f}$ and $0$ is F-semi-conical for $f$ if and only if it  
is F-semi-conical for $\tilde{f}$.
\end{itemize}
\end{prop}

\subsection{Normal forms for the single system case}\label{unpnmf}

\subsubsection{Conical intersection}\label{nfconical}
Define $f \in C^{\infty}(\R^2,\R^{2})$ such that $\chi(f)(0)\neq 0$.
In this case, $f$ is a diffeomorphism in a neighborhood of $0$.
Hence $f$ is right-equivalent to $\rm Id:\R^2 \to \R^{2}$.
The normal form provides the well-known Hamiltonian $H(u,v)=\begin{pmatrix}
u&v\\
v&-u
\end{pmatrix}$, for $u,v\in \R^2$.

\subsubsection{Semi-conical intersection}
The main result of this section is the following theorem.
\begin{teo}\label{nform}
Assume that $0$ is semi-conical for $f\in C^{\infty}(\R^2,\R)$.
Then $f$ is equivalent to $(u,v)\mapsto \begin{pmatrix}
h(u)u\\
u + v^2
\end{pmatrix}$ where $h:\R \to \R$ is a smooth function satisfying $h(0)=1$.
\end{teo}

\begin{oss}
As claimed in Proposition~\ref{sc:isolated}, it follows from the normal form of Theorem~\ref{nform} that semi-conical intersections are isolated (as eigenvalue intersections) in $\R^2$.
\end{oss}

The \textbf{algorithm} that we will refer as \textbf{(A)} to get the normal form is the following:
\begin{itemize}
\item \textbf{STEP 1:} By a left-equivalence we transform $f_1$ and $f_2$ into two functions $\tilde{f_1}$ and $\tilde{f_2}$ such that $\nabla \tilde{f_1}(0)=\nabla \tilde{f_2}(0)\neq 0$.
\item \textbf{STEP 2:} By a right-equivalence, we bring 
the non-conical direction to $\text{span}(e_2)$.
\item \textbf{STEP 3:} By a further right-equivalence then a time-equivalence 
we transform $f$ 
into the announced form.
\end{itemize}
\subsubsection{Proof of Theorem \ref{nform}: STEP 1}
\begin{prop}\label{normal}
Consider $f\in C^{\infty}(\R^2,\R^2)$ having a semi-conical intersection at $0$.
 Then there exists $\tilde{f}$ left-equivalent to $f$ 
such that $\nabla \tilde{f_1}(0)=\nabla \tilde{f_2}(0)\neq 0$.
\end{prop}
\begin{proof}
Without loss of generality, we can assume $\nabla f_1(0)\neq 0$.
Define $\alpha\in \R$ such that $\nabla f_2(0)=\alpha\nabla f_1(0)$.

Denote by 
$\tilde{f}$ the function obtained by applying to $f$ the 
left-equivalence associated with $\theta \in \mathbb{S}^1$ and $\zeta=1$ as in Remark~\ref{left}. 
Hence,
\[\nabla \tilde{f}_1=\cos(2\theta)\nabla f_1-\sin(2\theta)\nabla f_2,\quad \nabla \tilde{f}_2=\cos(2\theta)\nabla f_2+\sin(2\theta)\nabla f_1.\]
We have $\nabla \tilde{f_1}(0)=\nabla \tilde{f_2}(0)$ if and only if $\langle \begin{pmatrix}
\cos(2\theta)\\
\sin(2\theta)
\end{pmatrix}, \begin{pmatrix}
1-\alpha\\
-(1+\alpha)
\end{pmatrix}
\rangle=0$.
It is clearly possible to choose $\theta\in \mathbb{S}^1$ satisfying the previous condition, the proposition is proved.
\end{proof}

\subsubsection{Proof of Theorem \ref{nform}: STEP 2}

\begin{prop}\label{diff1}

Assume that $0$ is semi-conical for $f\in C^{\infty}(\R^2,\R^2)$.
There exists a right-equivalence $\phi:\R^2 \to \R^2$ of $f$ such that
$\tilde{f}=f\circ \phi$ satisfies $$\partial_2 \tilde{f}_1(0,0)=\partial_2 \tilde{f}_2(0,0)=0$$ and  $$\partial_1 \tilde{f}_1(0,0)\neq 0, \ \partial_1 \tilde{f}_2(0,0)\neq 0.$$
\end{prop}
\begin{proof}
Consider $r_1,r_2\neq 0$ and $\beta_1\in [0,2\pi]$ such that 
\[\partial_2 f_1(0,0)=r_1 \cos(\beta_1),\ \partial_1 f_1(0,0)=r_1\sin(\beta_1),\  \partial_2 f_2(0,0)=r_2 \cos(\beta_1),\ \partial_1 f_2(0,0)=r_2\sin(\beta_1).\]  
 Introducing the right-equivalence $\phi(u,v)=\begin{pmatrix}
-\sin(\beta_1) & \cos(\beta_1) \\
-\cos(\beta_1)&-\sin(\beta_1) \\
\end{pmatrix}\begin{pmatrix}
u\\
v
\end{pmatrix}
$ and $\tilde{f}
=(f\circ \phi)
$, 
 
we have
$D\tilde{f}(0,0)=\begin{pmatrix}
-r_1&0\\
-r_2&0
\end{pmatrix}$.
\end{proof}

Propositions \ref{normal} and \ref{diff1} lead us to consider the next condition:
\begin{equation}\label{SC2}
\tag{\textbf{SC}}
 f(0)=0,\ \partial_2 f(0)=0,\ \partial_1 f_1(0)=\partial_1 f_2(0)\neq 0,\ \partial_2\chi(f)(0)\neq 0.
\end{equation}

\subsubsection{Proof of Theorem \ref{nform}: STEP 3}

\begin{prop} \label{normalnc}
Let $f\in C^{\infty}(\R^{2},\R^2)$
satisfy Condition (\ref{SC2}).
Then there exists $h\in C^{\infty}(\R,\R)$ satisfying $h(0)=1$ such that $f$ is right-equivalent to
$(u,v)\mapsto \begin{pmatrix}
h(u)u\\
 u + v^2
 \end{pmatrix}$
 or $(u,v)\mapsto \begin{pmatrix}
h(u)u\\
 u - v^2
 \end{pmatrix}$.
\end{prop}
\begin{proof}
Because of the condition $\partial_1f_1(0)\neq 0$, 
the map $\Phi:(u,v)\mapsto (f_1(u,v),v)$ is a diffeomorphism in a neighborhood of $0$ and $g=f\circ \Phi^{-1}$ is right-equivalent to $f$.
Locally near $0$ we have 
\begin{align*}
g_1(u,v)&=u,\qquad g_2(u,v)=f_2(G(u,v),v),
\end{align*}
 where $G$ is a smooth function
 satisfying $\partial_1 G(u,v)=\frac{1}{\partial_1 f_1(G(u,v),v)}$ and $\partial_2 G(u,v)=-\frac{\partial_2 f_1(G(u,v),v)}{\partial_1 f_1(G(u,v),v)}$. 
Hence,
\begin{align*}
\partial_1 g_2(u,v)&=\partial_1 f_2(G(u,v),v) \partial_1 G(u,v)
,\\
\partial_2 g_2(u,v)&=\partial_1 f_2(G(u,v),v) \partial_2 G(u,v)+\partial_2 f_2(G(u,v),v).
\end{align*} 
The condition $\partial_1 f_1(0)\neq 0\neq \partial_1 f_2(0)$ implies that  $\partial_1 g_2(0)\neq 0$.
Moreover,
\begin{align*}
\partial_{22}g_2(u,v) = & \partial_{22}G(u,v)\partial_{1}f_2(G(u,v),v)+\partial_{22}f_2(G(u,v),v)\\
&+\partial_{2}G(u,v)^2\partial_{11}f_2(G(u,v),v)+2\partial_{2}G(u,v)\partial_{12}f_2(G(u,v),v).
\end{align*}
Evaluating at $0$
and noticing that $\partial_1 G(0)=\frac{1}{\partial_1 f_1(0)}$ and $\partial_2 G(0)=0$, we get $\partial_2 g_2(0)=0$ and 
\[\partial_{22} g_2(0)=\partial_{22}G(0)\partial_{1}f_2(0)+\partial_{22}f_2(0)=\frac{\partial_2 \chi(f)(0)}{\partial_1 f_1(0)}\neq 0,\]
since $f$ satisfies Condition (\ref{SC2}).
By the Implicit Function Theorem, $\partial_2 g_2(u,v)$ vanishes on a smooth curve 
$v=\eta(u)$ in a neighborhood of $0$. 
By Lemma~\ref{coromal}, there exist two smooth functions $m\in C^{\infty}(\R^2,\R)$ and $f_0\in C^{\infty}(\R,\R)$ such that
$g_2(u,v)=(v-\eta(u))^2 m(u,v)+f_0(u)$ in a neighborhood of $0$.
The conditions $\partial_1 g_2(0)\neq 0$ 
and $\partial_{22} g_2(0)\neq 0$ yield
$f_0'(0)\neq 0$ and $m(0)\neq 0$, respectively.
Applying first the right-equivalence $(u,v)\mapsto\begin{pmatrix}
u\\
(v-\eta(u))\sqrt{|m(u,v)|}
\end{pmatrix}$ and then
$(u,v)\mapsto\begin{pmatrix}
f_0(u)\\
v
\end{pmatrix}$,
we deduce that $f$ is right-equivalent to
$(u,v)\mapsto
\begin{pmatrix}
h(u)u\\
u+\text{sign}(m(0,0))v^2
\end{pmatrix}$,
for some smooth function $h$ obtained by inversion of $f_0$.
Noticing that the quantity $\frac{\partial_1 f_1(0)}{\partial_1 f_2(0)}$ is invariant by right-equivalence, this provides $h(0)=\frac{\partial_1 f_1(0)}{\partial_1 f_2(0)}=1$.
\end{proof}

\begin{proof}[Proof of Theorem~\ref{nform}] 
According to Propositions \ref{diff1} and \ref{normalnc}, we are left to prove
that
$$f:(u,v)\mapsto
\begin{pmatrix}
h(u)u \\
u-v^2
\end{pmatrix}$$ 
is right-time-equivalent to
$$(u,v)\mapsto
\begin{pmatrix}
\tilde{h}(u)u \\
u+v^2
\end{pmatrix},$$
where $\tilde{h}$ is in $C^\infty(\R,\R)$.
Indeed, we can apply the right-equivalence $(u,v)\mapsto (-u,v)$, then the time-equivalence associated with $\xi\equiv -1$. 
The result follows defining $\tilde{h}(u)=h(-u)$. 
\end{proof}

\subsection{Normal forms for the ensemble case}\label{pnmf}

Before discussing separately the conical and the semi-conical cases, let us present a useful technical result. 

\begin{lemma}\label{beta}
For every $f\in C^{\infty}(\R^3,\R^2)$ 
such that
$\partial_3 f_2(0)\neq 0$,
let $\beta(f)=\frac{\partial_3 f_1}{\partial_3 f_2}(0)$.
Then $\beta(f)$ is invariant by right-equivalence.
\end{lemma}
\begin{proof}
Let $\tilde{f}$ be right-equivalent to $f$ 
and let $\phi_1 \in C^{\infty}(\R^2,\R),\phi_2\in C^{\infty}(\R^2,\R),\phi_3\in C^{\infty}(\R,\R)$ be such that $\phi(u,v,z)=(\phi_1(u,v),\phi_2(u,v),\phi_3(z))$ 
is a right-equivalence between $f$ and $\tilde f$. 
Then,
\[\frac{\partial_3\tilde{f}_1}{\partial_3\tilde{f}_2}(0)=\frac{\partial_3 f_1(0)\phi_3'(0)}{\partial_3 f_2(0)\phi_3'(0)}=\beta(f),\]
using the fact $\phi_3'(0)\neq 0$ because $\phi$ is a diffeomorphism.
\end{proof}

\subsubsection{Conical case}
\begin{teo}
Let $f\in C^{\infty}(\R^3,\R^2)$.
 Then $0$ is F-conical for $f$ if and only if there exist $h_1,h_2\in C^{\infty}(\R^3,\R)$ 
 satisfying $h_1(0)=h_2(0)=1$, such that $f$ is equivalent to
$$(u,v,z)\mapsto \begin{pmatrix}
h_1(u,v,z)(z-u) \\
h_2(u,v,z)(z-v)
 \end{pmatrix}.$$
 \end{teo}
\begin{proof}
By the same argument as in the proof of Proposition \ref{normal}, there exists  $\tilde{f}$ left-equivalent to $f$ such that 
$\partial_3 \tilde{f}_1(0)\neq 0$, $\partial_3 \tilde{f}_2(0)\neq 0$ and $\beta(\tilde{f})=1$.
Using the fact that $\tilde{f}(0)=0$, we deduce that $\tilde{f}_1$ and $\tilde{f}_2$ 
vanish respectively on two smooth surfaces whose equations are of the form $z=\eta_1(u,v)$ and $z=\eta_2(u,v)$, where $\eta_1,\eta_2$ are smooth functions vanishing at $(0,0)$.
By Lemma~\ref{mal1}, there exist two smooth scalar functions $(u,v,z)\mapsto \phi_1(u,v,z),(u,v,z) \mapsto \phi_2(u,v,z)$ such that
\begin{align*}
\tilde{f}_1(u,v,z)&=\phi_1(u,v,z)(z-\eta_1(u,v)),\quad 
\tilde{f}_2(u,v,z)=\phi_2(u,v,z)(z-\eta_2(u,v)).
\end{align*}
Differentiating these expressions and evaluating them at $0$ we get that $\chi(\tilde{f})(0)=\phi_1(0)\phi_2(0)\chi(\eta)(0)$, where $\eta=(\eta_1,\eta_2)$.
Hence, by F-conicity of $0$, $\phi_1(0)\neq 0$, $\phi_2(0)\neq 0$, and $\chi(\eta)(0)\neq 0$.
In particular, $\eta$ is a diffeomorphism in a neighborhood of $(0,0)$.
Then $\tilde{f}$ is right-equivalent to 
\[\tilde{f}\circ \mu^{-1}(u,v,z)=\begin{pmatrix}
h_1(u,v,z)(u-z)\\
h_2(u,v,z)(v-z)
\end{pmatrix},\] 
where $\mu: (u,v,z)\mapsto \left(\eta_1(u,v), \eta_2(u,v),z \right)$ and $h_1,h_2\in C^{\infty}(\R^3,\R)$ satisfy $h_1(0)\neq 0$ and $h_2(0)\neq 0$.
By Lemma~\ref{beta}, $\frac{h_1(0)}{h_2(0)}=\beta(\tilde{f})=1$.
By applying a time equivalence associated with $\xi\equiv \frac{1}{h_1(0)}$,
 we conclude the proof of the theorem.
\end{proof}

\subsubsection{Semi-conical case}

\begin{teo}\label{hcpert}
Let $f\in C^{\infty}(\R^3,\R^2)$.
Then $0$ is F-semi-conical for $f\in C^{\infty}(\R^3,\R^2)$ if and only if there exist $h_1,h_2\in C^{\infty}(\R^3,\R^2)$ satisfying $h_1(0)=h_2(0)=1$ and $m \in C^{\infty}(\R,\R)$ satisfying $m(0)\notin \{-1,0\}$ such that $f$ is equivalent to
$$(u,v,z)\mapsto\begin{pmatrix}
h_1(u,v,z)(z-m(u)u) \\
h_2(u,v,z)(z+u+v^2)
\end{pmatrix}.$$
\end{teo}
Before proving the theorem, let us 
make some general considerations and
provide an intermediate result in Proposition~\ref{sc:2p}. 

First remark that, up to a left-equivalence, we can assume that 
\[
\partial_1 {f_1}(0)\neq 0, \ \partial_1 {f_2}(0)\neq 0,\ \partial_3 {f_1}(0)=\partial_3 {f_2}(0)\neq 0.
\]
In particular, $\beta({f})=1$.
In order to impose the non-conical direction to be in the $\text{span}(e_2)$-direction, we use the same right-equivalence of the plane $(u,v)$ as 
in the first step of the algorithm \textbf{(A)} in Section \ref{unpnmf} (see Proposition \ref{diff1}).
As a result, we end up with  $\hat{f}$ equivalent to $f$ and  
such that 
\begin{equation*}
{\begin{split}
& f(0)=0,\ \partial_2 f(0)=0,\ \partial_1 f_1(0)\neq 0,\ \partial_1 f_2(0)\neq 0,\ \partial_3 f_1(0)=\partial_3 f_2(0)\neq 0,\\
& \chi_{13}(f)(0)\neq 0,\ \partial_{2}\chi(f)(0) \neq 0.
\end{split}}
\eqno{({\rm SCP})}
 \end{equation*}
 Notice that the condition $\chi_{13}(f)(0)\neq 0$ can then be rewritten as $\partial_1 f_1(0) \neq \partial_1 f_2(0)$.

\begin{prop}\label{sc:2p}
Let $f\in C^\infty(\R^{3},\R^2)$ satisfy $\mathrm{(SCP)}$ at $0$. 
Then there exist  $h_1,h_2\in C^{\infty}(\R^3,\R)$ non-vanishing at $0$  and
$m \in C^{\infty}(\R,\R)$ such that $\frac{h_1(0)}{h_2(0)}=1$,  $m(0)=-\frac{\partial_1 f_1(0)}{\partial_1 f_2(0)}\notin \{-1,0\}$, and $f$ is right-equivalent to
\[(u,v,z)\mapsto\begin{pmatrix}
h_1(u,v,z)(z-m(u)u) \\
h_2(u,v,z)(z+u \pm v^2)
\end{pmatrix},\] 
where the sign depends on $f$.
\end{prop}
\begin{proof}
Using the fact that $f(0)=0$ and the conditions $\partial_3 f_1(0)\neq 0\ne \partial_3 f_2(0)$, we can deduce that $f_1$ and $f_2$ are smooth functions vanishing, in the neighborhood of the origin, on two smooth surfaces whose equations are, respectively, $z=\eta_1(u,v)$ and $z=\eta_2(u,v)$, where $\eta_1,\eta_2:\R^2 \to \R$ are smooth functions vanishing at $0$.
By Lemma~\ref{mal1}, there exist two smooth functions $\phi_1,\phi_2:\R^3 \to \R$ such that
\[f_1(u,v,z)=\phi_1(u,v,z)(z-\eta_1(u,v)), \qquad 
f_2(u,v,z)=\phi_2(u,v,z)(z-\eta_2(u,v)).
\]
Differentiating these expressions
with respect to $z$, we deduce that $\phi_1(0)\ne 0\ne \phi_2(0)$. Differentiating $f_1$ and $f_2$ with respect to $y$,
 we get then from $\mathrm{(SCP)}$ that
 $\partial_1 \eta_1(0)\ne 0\ne\partial_1 \eta_2(0)$
and 
$\partial_2 \eta_1(0)=\partial_2 \eta_2(0)=0$.
Applying the right-equivalence associated with the inverse of
$(u,v)\mapsto
\begin{pmatrix}
\eta_1(u,v)\\
v
\end{pmatrix}$,
we get that $f$ is right-equivalent to 
$$(u,v,z)\mapsto 
\begin{pmatrix}
\phi_1(G(u,v),v,z)(z-u)\\
\phi_2(G(u,v),v,z)(z-\eta_2(G(u,v),v))\\
\end{pmatrix},$$
for some smooth function $G$ such that $\partial_1G(0)\ne 0$.
Set $\tilde{\eta}(u,v)=\eta_2(G(u,v),v)$. Then 
$\tilde{\eta}(0)=0$ and
\[\partial_1 \tilde{\eta}(u,v)=\partial_1G(u,v) \partial_1 \eta_2(G(u,v),v),
 \quad
\partial_2 \tilde{\eta}(u,v)=\partial_1 \eta_2(G(u,v),v) \partial_2 G(u,v)+\partial_2 \eta_2(G(u,v),v).\]
Evaluating at zero, we get that $\partial_1 \tilde{\eta}(0)\ne 0$ and $\partial_2 \tilde{\eta}(0)=0$.
Differentiating once more and 
using the hypothesis $\partial_{2}\chi(f)(0) \neq 0 $, we have $\partial_{22} \tilde{\eta}(0)\neq 0$.
By the same arguments as in the proof of Theorem \ref{nform}, $f$ is right-equivalent to \[(u,v,z)\mapsto\begin{pmatrix}
h_1(u,v,z)(z-m(u)u) \\
h_2(u,v,z)(z+u \pm v^2)
\end{pmatrix}\] where $h_1,h_2\in C^{\infty}(\R^3,\R^2)$ and $m\in C^{\infty}(\R,\R)$.
Noticing that the quantities $\frac{\partial_1 f_1(0)}{\partial_1 f_2(0)}$ and $\beta(f)$ are invariant by right-equivalence, we get $
\frac{h_1(0)}{h_2(0)}=\beta(f)=1$ and $m(0)=-\frac{\partial_1 f_1(0)}{\partial_1 f_2(0)}\notin \{-1,0\}$ because $f$ satisfies .$\mathrm{(SCP)}$ at $0$.
\end{proof}

\begin{proof}[Proof of Theorem~\ref{hcpert}]
First notice that
if $f$ is of the form
$$f:(u,v,z)\mapsto
\begin{pmatrix}
h_1(u,v,z)(z-m(u)u) \\
h_2(u,v,z)(z+u-v^2)
\end{pmatrix}$$ 
then there exist $\tilde{h}_1,\tilde{h}_2\in C^{\infty}(\R^3,\R)$ and $\tilde m\in C^{\infty}(\R,\R)$ such that $f$ is right-time equivalent to
$$(u,v,z)\mapsto
\begin{pmatrix}
\tilde{h}_1(u,v,z)(z-\tilde{m}(u)u) \\
\tilde{h}_2(u,v,z)(z+u+v^2)
\end{pmatrix}.$$ 
Indeed, applying the right-equivalence $(u,v,z)\mapsto (-u,v,-z)$ and the time-equivalence associated with $\xi:(u,v)\mapsto -1$ 
the claim follows with $\tilde{h}_i(u,v,z)=h_i(-u,v,-z)$, $i\in \{1,2\}$, and $\tilde{m}(u)=m(-u)$.
Theorem~\ref{hcpert} hence follows from Proposition~\ref{sc:2p}.
\end{proof}

\section{Generic global properties of the singular locus \label{curve} }
\subsection{Proof of Lemma~\ref{lem:intro} and Theorem~\ref{thm:intro1}}

Let $f\in C^{\infty}(\R^{3},\R^2)$, $Z(f)=\{(u,v,z)\mid f(u,v,z)=(0,0)\}$, and denote by $Z_{\rm nc}(f)$ the set of non-conical intersections in $Z(f)$. Let $\pi(f)$ be the orthogonal projection of $Z(f)$ onto the plane $(u,v)$.

\begin{prop}\label{nc:tang}
Assume that $(\bar{u},\bar{v},\bar{z})$ is a F-semi-conical intersection for $f\in C^{\infty}(\R^3,\R^2)$.
Then $\pi(f)$ is tangent at $(\bar u,\bar v)$ to the non-conical direction of $f$ at $(\bar{u},\bar{v},\bar{z})$.
\end{prop}
\begin{proof}
By Proposition~\ref{trans}, $Z(f)$ is locally near $(\bar{u},\bar{v},\bar{z})$
a smooth curve that we parameterize by $c(t)=(u(t),v(t),z(t))_{t\in \R}$, 
with 
$c(0)=(\bar{u},\bar{v},\bar{z})$ and
$\dot{z}(0)=0$.
The condition $f(c(t))\equiv 0$ 
implies that $Df_{c({t})}(\dot{c}({t}))\equiv 0$. In particular, $(\dot{u}(0),\dot{v}(0))$ is collinear to $\eta=\begin{pmatrix}
-\partial_2 f_1(\bar{u},\bar{v},\bar{z})\\
\partial_1 f_1(\bar{u},\bar{v},\bar{z})
\end{pmatrix}
$, which is the non-conical direction of $f$ at $(\bar{u},\bar{v},\bar{z})$.
\end{proof}

\begin{prop}\label{cusp}
Assume that $f\in C^{\infty}(\R^3,\R^2)$ is a submersion at every point of $Z(f)$.
Then $\pi(f)$ has no cuspidal point.
\end{prop}
\begin{proof}
Fix a local smooth regular parametrization $c(t)=(u(t),v(t),z(t))_{t\in \R}$ of $c \subset 
Z(f)$. 
It is sufficient to show that there exist no $t\in\R$  such that  $\dot{u}(t)=\dot{v}(t)=0$.
By the same arguments as those used in the proof of Proposition~\ref{trans}, we get 
that $\dot{u}(t)=0$ implies that $\begin{vmatrix}
\partial_2 f_1(c(t))&\partial_3 f_1(c(t)) \\
\partial_2 f_2(c(t))&\partial_3 f_2(c(t))
\end{vmatrix}=0$, while 
$\dot{v}(t)=0$ implies that 
$\begin{vmatrix}
\partial_1 f_1(c(t))&\partial_3 f_1(c(t)) \\
\partial_1 f_2(c(t))&\partial_3 f_2(c(t))
\end{vmatrix}=0$.
If the two determinants  simultaneously vanish then  $f$ is not a submersion at $c(t)$.
\end{proof}

Propositions \ref{nc:tang} and \ref{cusp}, together with Remark~\ref{rmk:isolated}, prove 
Lemma~\ref{lem:intro}. 
As for Theorem~\ref{thm:intro1}, it follows from Proposition~\ref{cusp} together with 
Propositions~\ref{sub}, \ref{genconiq}, and \ref{ordre2}.

\subsection{Generic self-intersections of $\pi(f)$}

By a multi-jet version of the transversality arguments already used in the previous sections 
(see, for instance, \cite[\textsection{4, Theorem 4.13}]{Golub})
one can deduce the following result.

\begin{prop} \label{lem}
\begin{enumerate}
Generically with respect to $f\in C^{\infty}(\R^3,\R^2)$,
\item $\pi(f)$ has no triple points;
\item Let $(u,v)$ and $(\tilde{u},\tilde{v})$ be two double points of $\pi(f)$ and let $z_1\ne z_2$ and $\tilde z_1\ne \tilde z_2$ 
be 
such that $f(u,v,z_1)=f(u,v,z_2)=f(\tilde{u},\tilde{v},\tilde{z}_1)=f(\tilde{u},\tilde{v},\tilde{z}_2)=0$.
Then $z_i\neq \tilde{z}_j$ for every $i,j\in \{1,2\}$; 
\item Let $(u,v)$ 
and $z\neq \tilde{z}$ satisfy $f(u,v,z)=f(u,v,\bar{z})=0$. Then $(u,v,z)$ and $(u,v,\tilde{z})$ are F-conical for $f$;
\item Let $(u,v,z)$ and $(\tilde{u},\tilde{v},\tilde{z})$ be two non-conical intersections for $f$. Then $z\neq \tilde{z}$.
\end{enumerate}
\end{prop}

\section{Adiabatic control through a semi-conical intersection of eigenvalues \label{CONTROLSTRATEGY}}
\label{controtro}
Let us consider $f\in C^{\infty}(\R^2,\R^2)$, its associated Hamiltonian $H_f=\begin{pmatrix}
f_1&f_2\\
f_2&-f_1\\
\end{pmatrix}$, and a control path $(u(t),v(t))_{t\in [0,1]}$.
Denote by $\lambda^{-}(u,v)
$ and
$\lambda^{+}(u,v)
$
the smallest and the largest eigenvalue of $H_f(u,v)$, respectively.

In the following, 
we denote by $\phi_{-}(u(t),v(t))$ (respectively, $\phi_{+}(u(t),v(t))$) a real normalized eigenvector of $H_f(u(t),v(t))$ associated with $\lambda_{-}(u(t),v(t))$ (respectively, $\lambda_{+}(u(t),v(t))$). 
If $f(u(t),v(t))\ne 0$
then 
$\lambda_{-}(u(t),v(t))$ and $\lambda_{+}(u(t),v(t))$ are simple and
the choice of 
$\phi_{\pm}(u(t),v(t))$ is unique  up to multiplication  by $-1$.
If $(u(t),v(t))_{t\in [0,1]}$ does not cross $Z(f) $ 
then $t\mapsto \phi_{\pm}(u(t),v(t))$ 
and $t\mapsto \lambda^\pm(u(t),v(t))$
can be chosen with the same regularity as $(u(t),v(t))_{t\in [0,1]}$.
It is a classical fact that this may not be the case when $(u(t),v(t))_{t\in [0,1]}$ crosses 
$Z(f)$.
However, we are going to prove the existence of a $C^k$ basis of eigendirections of $H_f$ along a $C^{k+2}$ path $(u(t),v(t))_{t\in [0,1]}$ passing through a semi-conical intersection in a conical or a non-conical direction.

\subsection{Adiabatic dynamics}\label{adiabexplained}
Let $f\in C^{\infty}(\R^2,\R^2)$.
Consider a smooth regular control path $(u(t),v(t))_{t\in [0,1]}$ such that  
there exist $P\in C^2([0,1],{\rm SO}_2(\R))$ 
and  $\lambda\in C^{k}([0,1],\R)$
such that $\{ \lambda(t),-\lambda(t)\}$ is the spectrum of $H_f(u(t),v(t))$, and the columns of $P$ form a basis of eigenvectors of $H_f(u(t),v(t))$ for every $t\in [0,1]$.
We can write, for every $t\in [0,1]$,  $P(t)=\begin{pmatrix}
\cos(\theta(t))&-\sin(\theta(t))\\
\sin(\theta(t))&\cos(\theta(t))
\end{pmatrix}$ where $\theta\in C^2([0,1],\R
)$.

Let us study the dynamics of 
\begin{equation}\label{adupersch}
i \frac{d\psi_{\epsilon}(t)}{dt}=H_f(u(\epsilon t),v(\epsilon t)) \psi_{\epsilon}(t), \quad \psi_{\epsilon}(0)=\tilde{\psi}_0,
\end{equation}
where $t\in [0,\frac{1}{\epsilon}]$ and $\tilde{\psi}_0$ is independent of $\epsilon$.

Defining
$Y_{\epsilon}(\tau)=P(\tau)\psi_{\epsilon}(\frac{\tau}{\epsilon})$ for every $\tau\in [0,1]$, we have
\begin{equation} \label{transformeeunp}
i\frac{dY_{\epsilon}(\tau)}{d\tau}=\left(\frac{1}{\epsilon}\begin{pmatrix}
\lambda(\tau)&0\\
0&-\lambda(\tau)
\end{pmatrix}+
\begin{pmatrix}
0&i\dot{\theta}(\tau)\\
-i\dot{\theta}(\tau)&0
\end{pmatrix}\right)Y_{\epsilon}(\tau).
\end{equation} 
Thanks to the further change of variables $\tilde{Y}_{\epsilon}(\tau)=\begin{pmatrix}
e^{\frac{i}{\epsilon}\int_0^\tau\lambda(s)ds}&0\\
0&e^{-\frac{i}{\epsilon}\int_0^\tau\lambda(s)ds}
\end{pmatrix}
Y_{\epsilon}(\tau)$, the dynamics are transformed into 
\begin{equation} \label{reducedequnp}
\frac{d\tilde{Y}_{\epsilon}(\tau)}{d\tau}=\begin{pmatrix}
0&i\dot{\theta}(\tau)e^{\frac{2i}{\epsilon}\int_0^\tau\lambda(s)ds}\\
-i\dot{\theta}(\tau)e^{-\frac{2i}{\epsilon}\int_0^\tau\lambda(s)ds}&0
\end{pmatrix}\tilde{Y}_{\epsilon}(\tau).
\end{equation}
Based on Corollaries~\ref{avesti} and \ref{estimation}, one gets the following result.
\begin{teo}[Adiabatic Theorem]\label{adiabad}
Let $k\in \N$ and assume that $\lambda:[0,1] \to \R$ is $C^{k}$ 
and $\theta:[0,1] \to \R$ is $C^2$. 
Let $\psi_{\epsilon}:[0,1]\to\C^2$ be the solution of Equation~(\ref{adupersch}).
Assume that there exists $c>0$ such that 
 \begin{equation}\label{estimate:lambda}
\left|\int_0^t e^{\frac{2i}{\epsilon}\int_0^s\lambda(x)dx}ds\right|\leq c\epsilon^{1/(k+1)},\qquad \forall t \in [0,1].
\end{equation}
Then $\tilde{Y}_{\epsilon}(\tau)=\tilde{Y}_{\epsilon}(0)+O(\epsilon^{q})$ uniformly w.r.t. $\tau \in [0,1]$.
In particular, if $\tilde{\psi}_0=\begin{pmatrix}
\cos(\theta(0))\\
\sin(\theta(0))
\end{pmatrix}$, then $\psi_{\epsilon}(\frac{1}{\epsilon})=e^{i\eta}\begin{pmatrix}
\cos(\theta(1))\\
\sin(\theta(1))
\end{pmatrix}+O(\epsilon^{\frac{1}{k+1}})$, where $\eta$ possibly depends on $\epsilon$.
\end{teo}

\subsection{Regularity of the eigenpairs along smooth control paths}
The main goal of this section is to study the regularity of
eigenpairs of $H_f$
along smooth curves passing through a semi-conical intersection for $f$. Using the normal form obtained in Section~\ref{unpnmf}, we can restrict our attention to the case where $f$ has the form  $f:(u,v)\mapsto \begin{pmatrix} 
uh(u)\\
u+v^2
\end{pmatrix}$, where $h\in C^\infty(\R,\R)$ is such that $h(0)=1$.

\subsubsection{Conical directions}
We recall here the following regularity result which is a special case of \cite[Proposition 3.1]{ChittaroMason} and \cite[Lemma 3.2]{Ensemble}. 

\begin{prop}[Eigenpairs in the conical directions]\label{higherreg}
Consider a Hamiltonian $H_f$ where $f:(u,v)\mapsto \begin{pmatrix} 
uh(u)\\
u+v^2
\end{pmatrix}$ and $h\in C^\infty(\R,\R)$ is such that $h(0)=1$.
Let $\ell \in \N$, $t_0\in (0,1)$,  and $(u(t),v(t))_{t\in [0, 1]}$ be a  $C^{\ell+1}$ path such that $(u(t),v(t))=0$ if and only if $t=t_0$ 
and $\dot{u}(t_0)\neq 0$.
Define $\lambda_0,\lambda_1:[0,1]\to \R$ by
\begin{align*}
\lambda_0(t)&=\lambda_{-}(u(t),v(t)),\ \lambda_1(t)=\lambda_{+}(u(t),v(t)),\qquad \mbox{for } t<t_0,\\
\lambda_0(t)&=\lambda_{+}(u(t),v(t)),\ \lambda_1(t)=\lambda_{-}(u(t),v(t)),\qquad \mbox{for } t\ge t_0.
\end{align*}
Then $\lambda_0$ and $\lambda_1$ are $C^{\ell+1}$ on $[0,1]$. Moreover, there exist
$\Phi_0,\Phi_1\in C^\ell([0,1],\R^2)$ such that $\Phi_j(t)$ is a normalized eigenvector of $H_f(u(t),v(t))$ corresponding to the eigenvalue $\lambda_j(t)$ for $j\in\{0,1\}$ and $t\in [0,1]$.
\end{prop}

The following proposition
states that the limit eigenvectors along a $C^2$ curve crossing conically a semi-conical intersection do not depend on the choice of the curve.

\begin{prop}[Limit eigenvectors along a conical direction]\label{reg}
Consider $f$, $(u(t),v(t))_{t\in [0,1]}$, $t_0$, and $\Phi_0,\Phi_1$ as in Proposition~\ref{higherreg}, where the latter are uniquely defined  up to multiplication to $-1$.
Then $\Phi_0(t_0)$ and $\Phi_1(t_0)$ depend only on the sign of $\dot{u}(t_0)$.
More precisely, 
$\Phi_0(t_0)=
\frac{1}{\sqrt{1+\bar V
^2}}\begin{pmatrix}
-1\\
\bar V
\end{pmatrix}
$ and $\Phi_1(t_0)=
\frac{1}{\sqrt{1+\bar V^2
}}\begin{pmatrix}
\bar V
\\
1
\end{pmatrix}
$ with 
$\bar V
=-(1+{\rm sign}(\dot{u}(t_0))\sqrt{2})$. 
\end{prop}
\begin{proof}
By definition of $\Phi_0,\Phi_1$, we have 
\begin{align*}
\Phi_0(t)&=\phi_{-}(u(t),v(t)),\ \Phi_1(t)=\phi_{+}(u(t),v(t)),\qquad \mbox{for } t<t_0.
\end{align*}

Hence, up to multiplication by $-1$, 
\begin{align*}
\Phi_0(t)&=\frac{1}{\sqrt{1+V(t)^2}}\begin{pmatrix}
-1\\
V(t)
\end{pmatrix}
,\ \Phi_1(t)=\frac{1}{\sqrt{1+V(t)^2}}\begin{pmatrix}
V(t)\\
1
\end{pmatrix},\qquad \mbox{for } t<t_0,
\end{align*}
where $V(t)=0$ if $u(t)+v(t)^2=0$ and 
\[V(t)=\frac{-u(t)h(u(t))+ \sqrt{u(t)^2h(u(t))^2+(u(t)+v(t)^2)^2}}{u(t)+v(t)^2}\]
otherwise.
Then, as 
$t\to t_0^{-}$, 
\begin{align*}
V(t)&=\frac{-\dot{u}(t_0)h(0)(t-t_0)+o(t-t_0)+\sqrt{\dot{u}(t_0)^2 (t-t_0)^2(h(0)^2+1)+o((t-t_0)^2)}}{\dot{u}(t_0)(t-t_0)+o(t-t_0)}\\
&=-\frac{ \dot{u}(t_0)(t-t_0)+|t-t_0| |\dot{u}(t_0)|\sqrt{2}}{\dot{u}(t_0)(t-t_0)}    +o(1)\\
&=-(1+{\rm sign}(\dot{u}(t_0))\sqrt{2})+o(1).
\end{align*}
Since $\Phi_0,\Phi_1$ are continuous on $[0,1]$ by Proposition~\ref{higherreg}, the conclusion follows.  
\end{proof}

\subsubsection{Non-conical direction}
\begin{prop}[Continuity of the eigenstates in the non-conical direction]\label{regnc}
Let $f$ be 
as in Proposition~\ref{higherreg}.
Let $\ell\in\N$, $t_0\in (0,1)$,  and $(u(t),v(t))_{t\in [0, 1]}$ be a $C^{\ell+2}$ path such that $(u(t),v(t))=0$ if and only if $t=t_0$ 
and $\dot{u}(t_0)=0$ (i.e., $(u(t),v(t))$ passes through $0$ in the non-conical direction at $t=t_0$ as on Figure~\ref{nc:pass}).
Then there exist $\lambda_0,\lambda_1\in C^{\ell+2}([0,1],\R^2)$, $\Phi_0,\Phi_1\in C^\ell([0,1],\R^2)$ such that 
$\lambda_0(t)=\lambda_{-}(u(t),v(t))$, $\lambda_1(t)=\lambda_{+}(u(t),v(t))$, $\Phi_0(t)= \phi_{-}(u(t),v(t))$ and $\Phi_1(t)= \phi_{+}(u(t),v(t))$ for every $t\in [0,t_0)\cup (t_0,1]$. 
Moreover, defining $\beta=\frac{\ddot{u}(t_0)}{2}+\dot{v}(t_0)^2$, we have $\Phi_0(t_0)= \frac{1}{\sqrt{1+\bar V^2}}\begin{pmatrix}
-1\\
\bar V
\end{pmatrix}$ and $\Phi_1(t_0)=\frac{1}{\sqrt{1+\bar V^2}}\begin{pmatrix}
\bar V\\
1
\end{pmatrix}$, 
where 
$\bar V=-\frac{\frac{\ddot{u}(t_0)}{2}-\sqrt{\frac{\ddot{u}(t_0)^2}{4}+\beta^2}}{\beta}$ if $\beta \neq 0$. If $\beta= 0$, we have 
$\Phi_0(t_0)=\begin{pmatrix}
1\\
0
\end{pmatrix}$ and
$\Phi_1(t_0)=\begin{pmatrix}
0\\
1
\end{pmatrix}$.
\end{prop}
\begin{proof}
The condition $\dot{u}(t_0)=0$ provides $u(t)+v(t)^2=(t-t_0)^2(\frac{\ddot{u}(t_0)}{2}+\dot{v}(t_0)^2)+o((t-t_0)^2)=\beta (t-t_0)^2+o((t-t_0)^2)$ when $t\to t_0$.

By direct computations, we show that $\frac{d^2\lambda_{\pm}(u(t),v(t))}{dt^2}=\pm \sqrt{\beta^2+\ddot{u}(t_0)^2}+o(1)$ when $t\to t_0$.
Hence $t\mapsto \lambda_{\pm}(u(t),v(t))$ can be extended in a $C^2$ function at $t=t_0$ by fixing $\frac{d^2\lambda_{\pm}(u(t),v(t))}{dt^2}=2\pm \sqrt{\beta^2+\ddot{u}(t_0)^2}$. The $C^{l+2}$ regularity follows by higher order analog computations.

Define, for every $t\in [0,1]$ such that $u(t)+v(t)^2\ne 0$, $V(t)=\frac{-u(t)h(u(t))+ \sqrt{u(t)^2h(u(t))^2+(u(t)+v(t)^2)^2}}{u(t)+v(t)^2}$. 
Reasoning as in the proof of Proposition~\ref{reg}, we must prove that $V$ can be extended as a $C^\ell$ function at $t_0$ by setting $V(t_0)=\bar V$.

\textbf{First case: $\beta\neq 0$.}
Assuming that $\beta \neq 0$, 
we have, as $t\to t_0$,
 \[V(t)=-\frac{\frac{\ddot{u}(t_0)}{2} (t-t_0)^2-\sqrt{\frac{(t-t_0)^4}{4}\ddot{u}(t_0)^2+\beta^2 (t-t_0)^4+o((t-t_0)^4)}}{\beta (t-t_0)^2+o((t-t_0)^2)}=-\frac{\frac{\ddot{u}(t_0)}{2}-\sqrt{\frac{\ddot{u}(t_0)^2}{4}+\beta^2}}{\beta}+o(1). 
 \]
 The continuity of $V$ is proved in the case $\beta\neq 0$.
 The same computations show that $V$ is $C^\ell$ at $t_0$ if $(u(t),v(t))_{t\in [0,1]}$ is $C^{\ell+2}$.

\textbf{Second case: $\beta=0$.}  
 In the case $\beta=0$, consider the left equivalence $\tilde{f}$ of $f$ associated with $\theta=\frac{\pi}{4}$ and $\zeta=-1$ as in Remark~\ref{left} , so that, we have $H_{\tilde{f}}=P_{\theta,\zeta}H_f P_{\theta,\zeta}^{-1}$. If $f=(f_1,f_2)\in C^{\infty}(\R^2,\R^2)$, then we have $\tilde{f}=(f_2,f_1)$.
Define, for every $t\in [0,1]$ such that $u(t)h(u(t))\ne 0$, $\tilde{V}(t)=\frac{-(u(t)+v(t)^2)+ \sqrt{u(t)^2h(u(t))^2+(u(t)+v(t)^2)^2}}{u(t)h(u(t))}$. We have easily $\lim_{t\to t_0}\tilde{V}(t)=-1$.
Hence we can define continuous eigenvectors $\tilde{\Phi_0}$ and $\tilde{\Phi_1}$ of $H_{\tilde{f}}$ along $(u(t),v(t))_{t\in [0,1]}$, respectively equal, up to phases, to  $\frac{1}{\sqrt{2}}\begin{pmatrix}
1\\
1
\end{pmatrix}$ and $\frac{1}{\sqrt{2}}\begin{pmatrix}
-1\\
1
\end{pmatrix}$ at $t=t_0$.
Knowing that the eigenvectors of $H_f$ are equal, up to phases, to $P_{\theta,\zeta} \tilde{\Phi_j}$, for $j\in \{0,1\}$, we can deduce the continuity of $\Phi_0$ and $\Phi_1$ at $t=t_0$, and that $\Phi_0(t_0)=\begin{pmatrix}
1\\
0
\end{pmatrix}$ and
$\Phi_1(t_0)=\begin{pmatrix}
0\\
1
\end{pmatrix}$. 
\end{proof}

\subsection{Dynamical properties at semi-conical intersections of eigenvalues}
By using the previous results, we get the following adiabatic approximations along curves going through a semi-conical intersection, either along conical directions (Proposition~\ref{conicalfull}) or along the non-conical direction (Proposition~\ref{errornc}).
\begin{prop}
\label{conicalfull}
Let $f$ 
and $(u(t),v(t))_{t\in [0,1]}$ be as in Proposition~\ref{higherreg}.
Consider a solution $\psi_{\epsilon}$ of Equation~(\ref{adupersch}) such that $\tilde\psi_{0}=\phi_{-}(u(0),v(0))$.
Then $\langle \psi_{\epsilon}(\frac{1}{\epsilon}), \phi_{-}(u(1),v(1))\rangle=O(\sqrt{\epsilon})$.
\end{prop}
\begin{proof}
First notice that $\lambda_{-}$ and $\lambda_{+}$ are $C^2$ separately on $[0,t_0]$ and $[t_0,1]$. Moreover, Corollary~\ref{estimation} proves that they satisfy (\ref{estimate:lambda}) with $k=1$
on $[0,t_0]$ and $[t_0,1]$.
By  Proposition~\ref{reg}, moreover, 
there exists a $C^2$ basis of eigenvectors 
defined on $[0,1]$ with $\lim_{t\to t_0^-}\phi_\pm(u(t),v(t))=\lim_{t\to t_0^+}\phi_\mp(u(t),v(t))$. 
By applying Theorem~\ref{adiabad} on the interval $[0,t_0]$, then on the interval $[t_0,1]$, we get the result.
\end{proof}

\begin{prop}\label{errornc}
Let $f$ and $(u(t),v(t))_{t\in [0,1]}$ be as in Proposition~\ref{regnc} with $l\geq 2$ (see Figure~\ref{nc:pass}).
Consider a solution $\psi_{\epsilon}$ of Equation~(\ref{adpersch}) such that $\psi_{\epsilon}(0)=\phi_{-}(u(0),v(0))$.
Then $\langle \psi_{\epsilon}(\frac{1}{\epsilon}), \phi_{+}(u(1),v(1))\rangle =O(\epsilon^{1/3})$.
\end{prop}
\begin{proof}
Let $\lambda_0$ be as in Proposition~\ref{regnc} and define $\varphi(\tau)=\int_0^{\tau}\lambda_0(s)ds$ for $\tau\in [0,1]$. Notice that $\lambda_0$ is at least $C^2$ by Proposition~\ref{regnc} and that $\varphi$ 
satisfies $\varphi(t_0)=\varphi'(t_0)=\varphi''(t_0)=0$ and $\varphi^{(3)}(t_0)\neq 0$.
Hence, by Lemma~\ref{vdc}, $\lambda_0$ satisfies the estimate (\ref{estimate:lambda}) with $k=2$.
The result follows by applying Theorem~\ref{adiabad} in combination with Proposition~\ref{regnc}.
\end{proof}

\begin{figure}[h!] \label{ncpass}
\center
        \includegraphics[scale=0.9]{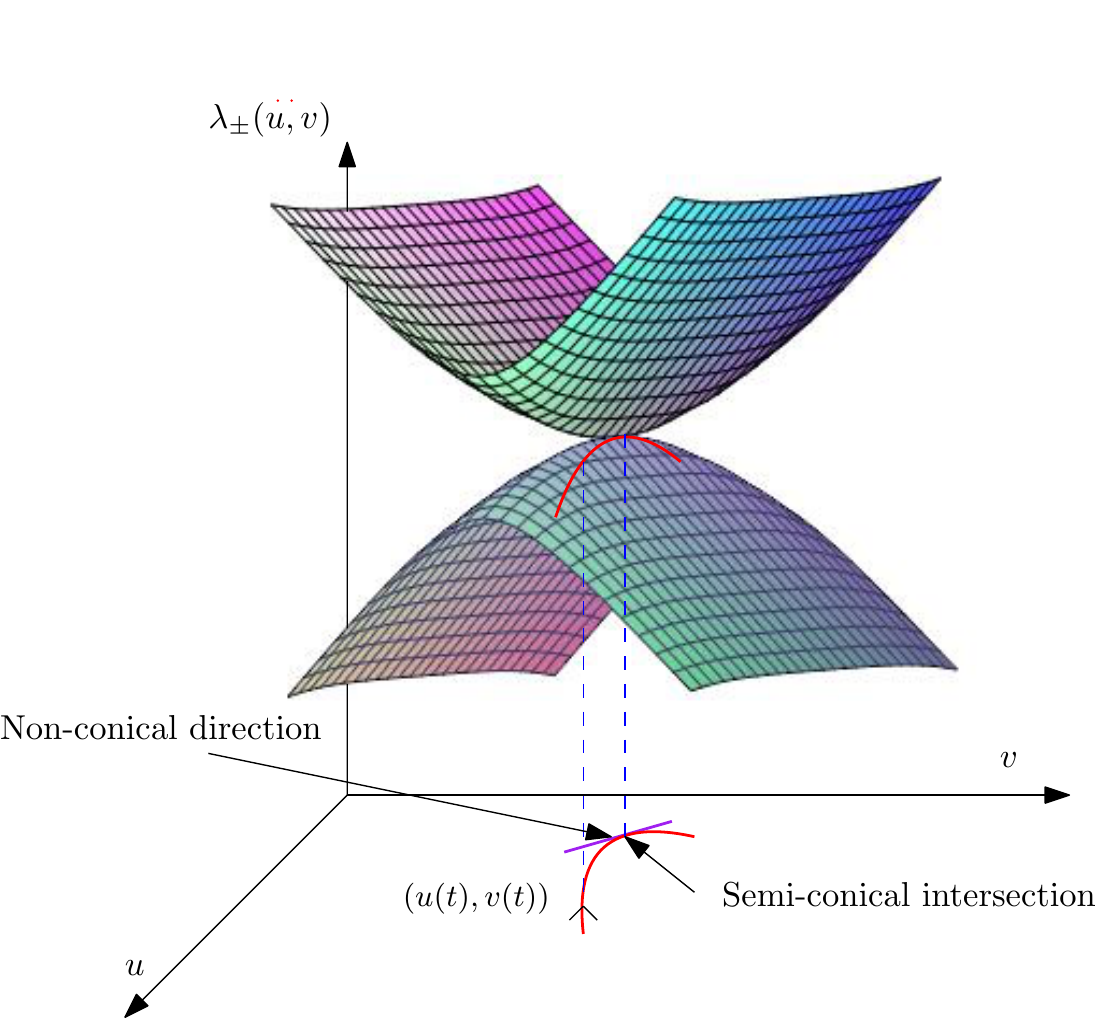}  
        \caption{\label{nc:pass}A control path passing at a semi-conical intersection in the non-conical direction as a function of the controls $(u,v)\in \R^2$. }
\end{figure}

\section{Control of an ensemble of systems  \label{contiti}}
The main goal of this section is to prove the controllability result stated in Theorem~\ref{thm:intro2}.
\subsection{Ensemble adiabatic dynamics}
Let $f\in C^{\infty}(\R^3,\R^2)$. 
Consider a smooth regular control path $(u(t),v(t))_{t\in [0,1]}$. 
In analogy with the previous sections, denote by  $\lambda^{\pm}_z(u(t),v(t))
$ the eigenvalues of $H_f(u(t),v(t),z)$. 
Similarly, let $\phi_{\pm}^z(u(t),v(t))$ 
be two real normalized eigenvector of $H_f(u(t),v(t),,z)$ at $(u(t),v(t))$ associated with $\lambda^{\pm}_z(u(t),v(t))$, 
uniquely defined up to a sign.

Let ${\cal Z}=[z_0,z_1]$ be a compact interval of $\R$. 
 Assume that  
for every $z\in {\cal Z}$ there exist $P_z\in C^1([0,1],{\rm SO}_2(\R))$ and  $\lambda_z\in C^{1}([0,1],\R)$ such that 
$\{ \lambda_z(t),-\lambda_z(t)\}$ is the spectrum of $H_f(u(t),v(t),z)$ and the columns of $P_z$ form a basis of eigenvectors of $H_f(u(t),v(t),z)$ for every $t\in [0,1]$.
We can write, for every $t\in [0,1]$,  $P_z(t)=\begin{pmatrix}
\cos(\theta_z(t))&-\sin(\theta_z(t))\\
\sin(\theta_z(t))&\cos(\theta_z(t))
\end{pmatrix}$ where $\theta_z\in C^1([0,1],\R)$.

Let us study the dynamics of 
\begin{equation}\label{adpersch}
i \frac{d\psi_{\epsilon}^z(t)}{dt}=H_f(u(\epsilon t),v(\epsilon t),z) \psi_{\epsilon}^z(t), \quad \psi_{\epsilon}^z(0)=\tilde{\psi_0^z},
\end{equation}
where $t\in [0,\frac{1}{\epsilon}]$ and $\tilde{\psi_0^z}$ is independent of $\epsilon$.

Defining
$Y_{\epsilon}^z(\tau)=P_z(\tau)\psi_{\epsilon}^z(\frac{\tau}{\epsilon})$ for every $\tau\in [0,1]$, we have
\begin{equation} \label{ad:P}
i\frac{dY_{\epsilon}^z(\tau)}{d\tau}=\left(\frac{1}{\epsilon}\begin{pmatrix}
\lambda_z(\tau)&0\\
0&-\lambda_z(\tau)
\end{pmatrix}+
\begin{pmatrix}
0&i\dot{\theta}_z(\tau)\\
-i\dot{\theta}_z(\tau)&0
\end{pmatrix}\right)Y_{\epsilon}^z(\tau).
\end{equation} 
Thanks to the further change of variable $\tilde{Y_{\epsilon}^z}(\tau)=\begin{pmatrix}
e^{\frac{i}{\epsilon}\int_0^\tau\lambda_z(s)ds}&0\\
0&e^{-\frac{i}{\epsilon}\int_0^\tau\lambda_z(s)ds}
\end{pmatrix}
Y_{\epsilon}^z(\tau)$, the dynamics is transformed into 
\begin{equation} \label{reducedeqpert}
\frac{d\tilde{Y_{\epsilon}^z}(\tau)}{d\tau}=\begin{pmatrix}
0&i\dot{\theta}_z(\tau)e^{\frac{2i}{\epsilon}\int_0^\tau\lambda_z(s)ds}\\
-i\dot{\theta}_z(\tau)e^{-\frac{2i}{\epsilon}\int_0^\tau\lambda_z(s)ds}&0
\end{pmatrix}\tilde{Y_{\epsilon}^z}(\tau).
\end{equation}
Based on Corollary~\ref{estimationensemble}, one get the following result.
\begin{teo}[Parametric Adiabatic Theorem]\label{Unifadiab}
Let $k\in \N$. For every $z\in [z_0,z_1]$, assume that $\lambda_z:[0,1] \to \R$ is $C^{k}$ in $[0,1]$ and $\theta_z$ is $C^2$ in $[0,1]$.
Let $\psi_{\epsilon}^z(t)$ be the solution of Equation~(\ref{adpersch}).
Assume that there exists $c>0$ such that for every $t$ in $[0,1]$, and for every $z\in [z_0,z_1]$,
 \begin{equation}\label{estimate:lambdaz}
\left|\int_0^t e^{\frac{2i}{\epsilon}\int_0^s\lambda_z(x)dx}ds\right|\leq c\epsilon^{\frac1{k+1}},\qquad \forall\epsilon>0,
\end{equation}
and that $(t,z)\mapsto \dot{\theta}_z(t)$ and $(t,z)\mapsto \ddot{\theta}_z(t)$ are uniformly bounded with respect to $(t,z)\in [0,1]\times[z_0,z_1]$.
Then we have $\tilde{Y_{\epsilon}}(\tau)=\tilde{Y_{\epsilon}^z}(0)+O(\epsilon^{\frac{1}{k+1}})$, uniformly w.r.t. $\tau \in [0,1]$ and $z\in [z_0,z_1]$.
In particular, if $\tilde{\psi_0^z}=\begin{pmatrix}
\cos(\theta_z(0))\\
\sin(\theta_z(0))
\end{pmatrix}$, then we have $\psi_{\epsilon}^z(\frac{1}{\epsilon})=e^{i\eta}\begin{pmatrix}
\cos(\theta_z(1))\\
\sin(\theta_z(1))
\end{pmatrix}+O(\epsilon^{\frac{1}{k+1}})$, uniformly w.r.t. $z\in [z_0,z_1]$, where $\eta$ is possibly depending on $\epsilon,z$.
\end{teo}

\subsection{Controllability properties between the eigenstates for the normal forms}
Let $f:(u,v,z)\mapsto 
\begin{pmatrix}
h_1(u,v,z)(z-m(u)u)\\
h_2(u,v,z)(z+u+v^2)
\end{pmatrix}
$, where $h_1,h_1\in C^{\infty}(\R^3,\R)$ and $m\in C^{\infty}(\R,\R)$ satisfy 
$h_1(0)=h_2(0)\ne 0$ and 
 $m(0)\notin \{-1,0\}$.  
Recall that $f$ has a F-semi-conical intersection at $0$.

Consider a compact neighborhood $\mathcal{S}$  of $0$ in $\R^3$ on which 
the product $h_1 h_2$ does not vanish and $m$ is different from $0$ and $-1$.
Assume that $\mathcal{S}$ writes $\mathcal{S}=U \times [z_0,z_1]$,  where $U$ is a compact neighborhood of $0$ in $\R^2$ and $z_0<0<z_1$. 
Define $\mathcal{C}=Z(f)\cap \mathcal{S}$ and, for every $(u,v,z)\in \mathcal{S}$, $h(u,v,z)=\frac{h_1(u,v,z)}{h_2(u,v,z)}$.
Notice that 
$(u,v,z)\in \mathcal{S}$ is in $\mathcal{C}$ if and only if 
$m(u)u=z=-u-v^2$.
Up to restricting $U$, we can assume that $u\to (m(u)+1)u$ is monotone, so that 
\begin{equation}\label{pif}
(m(u)+1)u=-v^2
\end{equation}
 defines a smooth submanifold of $U$ which contains the projection 
$\pi(\mathcal{C})$ of $\mathcal{C}$ onto the plane of controls $(u,v)\in \R^2$.

Without loss of generality, assume that $m(0)>-1$ (the case $m(0)<-1$ being analogous). 
According to \eqref{pif}, this means that $\pi(\mathcal{C})$ lies in the intersection of $U$ with the left half-plane. Notice that the sign of $z=m(u)u$ on $\mathcal{C}$ is the opposite as the sign of $m(0)$.

\subsubsection{Uniform adiabatic estimates when  $(u(t),v(t))_{t\in [0,1]}\subset \pi(\mathcal{C})$}
 Assume that $(u,v):[0,1]\to U$ is a regular $C^{\infty}$ control path satisfying the following conditions, referred to as (\textbf{C}):
\begin{itemize}
\item $(u,v)\subset \pi(\mathcal{C})$
\item $(u,v)(0)=(u_0,v_0)
$ where $(u_0,v_0,z_0)\in \mathcal{C}$  is a F-conical intersection for $f$;
\item  $(u,v)(1)=(0,0)$.
\end{itemize}
Under the previous assumptions, for every $z\in [z_0,0]$, we can consider $t_z$ as the unique element in $[0,1]$ that satisfies $(u(t_z),v(t_z),z)\in \mathcal{C}$.
On the other hand,  for  $z\in (0,z_1]$ there exist no $t\in [0,1]$ such that $(u(t),v(t),z)\in \mathcal{C}$.
By the regularity of $(u,v)$, 
the application $[z_0,0]\ni z \mapsto t_z$ is $C^\infty$.
Moreover, as a direct consequence of equation~\eqref{pif}, which holds on $\mathcal{C}$, we have 
$f_1(u(t),v(t),z)=h(u(t),v(t),z)f_2(u(t),v(t),z)$ for every $t\in [0,1]$ and $z\in [z_0,z_1]$. 
\begin{defi}\label{v:z}
For any $z\in [z_0,z_1]$, define 
$V_z:[0,1]\to \R$ as follows: 
if $z\in [z_0,0)$, let $V_z(t)=-(h(u(t),v(t),z)-\sqrt{1+h(u(t),v(t),z)^2})$ for $t<t_z$ and 
 $V_z(t)=-(h(u(t),v(t),z)+\sqrt{1+h(u(t),v(t),z)^2})$ for $t\geq t_z$; if $z\in [0,z_1]$, let $V_z(t)=-(h(u(t),v(t),z)-\sqrt{1+h(u(t),v(t),z)^2})$ for every $t\in [0,1]$.
 \end{defi}
 
By Propositions~\ref{higherreg} and \ref{regnc}, we get the following result on the regularity of eigenpairs.
\begin{prop}[Regularity of eigenpairs for a control path $(u,v)\subset \pi(\mathcal{C})$]\label{regensemble}
Let $(u,v)$ satisfy {\bf (C)}.
For $z\in [0,z_1]$, define, for every $t\in [0,1]$, $\Phi_0^z(t)=\phi_{-}(u(t),v(t))$, $\Phi_1^z(t)=\phi_{+}(u(t),v(t))$, $\lambda_0^z(t)=\lambda_{-}^z(u(t),v(t))$, and $\lambda_1^z(t)=\lambda_{+}^z(u(t),v(t))$.
For $z\in [z_0,0)$, define $\lambda_0^z,\lambda_1^z:[0,1]\to \R$ by
\begin{align*}
\lambda_0^z(t)&=\lambda_{-}^z(u(t),v(t)),\ \lambda_1^z(t)=\lambda_{+}^z(u(t),v(t)),\qquad \mbox{for } t<t_z,\\
\lambda_0^z(t)&=\lambda_{+}^z(u(t),v(t)),\ \lambda_1^z(t)=\lambda_{-}^z(u(t),v(t)),\qquad \mbox{for } t\ge t_z,
\end{align*}
and $\Phi_0^z, \Phi_1^z:[0,1]\to \R^2$ by 
\begin{align*}
\Phi_0^z(t)&=\phi_{-}^z(u(t),v(t)),\ \Phi_1^z(t)=\phi_{+}^z(u(t),v(t)),\qquad \mbox{for } t<t_z,\\
\Phi_0^z(t)&=\phi_{+}^z(u(t),v(t)),\ \Phi_1^z(t)=\phi_{-}^z(u(t),v(t)),\qquad \mbox{for } t\ge t_z.
\end{align*}
Then, for every $z\in [z_0,z_1]$, $\lambda_0^z,\lambda_1^z$ and $\Phi_0^z,\Phi_1^z$ are $C^{\infty}$ on $[0,1]$. 
Moreover, $\Phi_0^z$ and $\Phi_1^z$ can be written as $\Phi_0^z=
\frac{1}{\sqrt{1+ V_z
^2}}\begin{pmatrix}
-1\\
 V_z
\end{pmatrix}
$ and $\Phi_1^z=
\frac{1}{\sqrt{1+V_z^2
}}\begin{pmatrix}
 V_z
\\
1
\end{pmatrix}
$, where 
$V_z\in C^{\infty}([0,1],\R)$ is defined as in Definition~\ref{v:z}.
\end{prop}
A direct corollary of Proposition~\ref{regensemble} is the following. 
\begin{prop}\label{bound}
For any $z\in [z_0,z_1]$, let $\theta_z=\arctan(V_z)\in C^{\infty}([0,1],\R)$, where $V_z$ is defined as in Definition~\ref{v:z}.
Then $(t,z)\mapsto \dot{\theta}_z(t)$ and $(t,z)\mapsto \ddot{\theta}_z(t)$ are bounded w.r.t. $(t,z)\in [0,1]\times [z_0,z_1]$. 
\end{prop}
\begin{proof}
By definition of $V_z$, for every $(t,z)\in [0,1]\times [z_0,z_1]$, $\dot{\theta}_z(t)$ and $\ddot{\theta}_z(t)$  depend only on $\frac{d}{dt} h(u(t),v(t),z)$ and $\frac{d^2}{dt^2} h(u(t),v(t),z)$, which are uniformly bounded w.r.t. $(t,z)\in [0,1] \times [z_0,z_1]$  because 
$h\in C^{\infty}(\mathcal{S},\R)$.  
\end{proof}

\begin{oss}
In the particular (non-generic) case in which $h$ is constant, the curve $(u,v)$ is non-mixing for all $z\in [z_0,z_1]$, in the sense developed in~\cite{Bos}. 
Non-mixing curves are characterized by an enhanced adiabatic approximation with respect to general curves passing through an eigenvalue intersection.  
\end{oss}

\begin{prop}\label{esti2}
The functions $\lambda_{0}^z,\lambda_{1}^z$, defined as in Proposition~\ref{regensemble}, satisfy (\ref{estimate:lambdaz}) with $k=2$.
\end{prop}
\begin{proof}
As a first step of the proof, let us show the following local estimate: There exist $t_1\in [0,1)$,  a nonempty compact neighborhood $W \subset [z_0,z_1]$ of $0$, and $C_1>0$ independent of $z$ such that for every $t\in[t_1,1]$, $j\in\{0,1\}$, and $z\in W$ we have 
\begin{equation}\label{local-estimate}
\left|\int_{t_1}^t e^{\frac{2i}{\epsilon}\int_0^s\lambda_j^z(r)dr}ds\right|\leq C_1 \epsilon^{\frac1 3},\qquad \forall\epsilon>0.
\end{equation}
According to
 Corollary~\ref{estimationensemble}, it is enough to prove that there exist $t_1$ and $W$ as above such that 
 \begin{equation}\label{pluslocal}
|\ddot{\lambda}_z(t)|>c,\qquad\forall z\in W,\;t\in[t_1,1]\setminus\{t_z\},
 \end{equation}
where $c>0$ is independent of $z\in W$.
Notice that $\lambda_z(t)=|z-m(u(t))u(t)|\sqrt{h_1(u(t),v(t),z)^2+h_2(u(t),v(t),z)^2}$.
By hypothesis {\bf (C)}, $u(1)=\dot u(1)=0$ and $\ddot u(1)=0$. Hence,
\[
\left.\frac{d}{dt}(z-m(u(t))u(t))\right|_{t=1}=0,
\quad 
\left.\frac{d^2}{dt^2}(z-m(u(t))u(t))\right|_{t=1}=(z-m(0)) \ddot u(1),
\]
and, in particular, for $z=0$ we have 
\[  \lim_{z\to 0}\left|\frac{d^2}{dt^2}\lambda_z(t)\right|_{t=1} =|m(0) \ddot u(1)|>0.\]
Inequality \eqref{pluslocal}, and hence the required local estimate \eqref{local-estimate}, follow by a continuity argument. 
Notice that, up to restricting $W$ or increasing $t_1$, we can assume that $\{t_z\mid z\in W\cap[z_0,0]\}=[t_1,1]$.

Let us now extend \eqref{local-estimate} to $z\in [z_0,z_1]$ and $t\in [0,1]$.
For $z\in [0,z_1]\setminus W$, there exists $c_1>0$ (independent of $z$) such that 
$|\lambda_0^z(t)|>c_1>0$ for every $t\in [0,1]$. Hence, by applying Lemma~\ref{k1ens}, we have $|\int_0^t e^{\frac{2i}{\epsilon}\int_0^s\lambda_0^z(r)dr}ds|\leq C_1 \epsilon$, where $C_1>0$ is independent of $(t,z)\in [0,1] \times ([0,z_1]\setminus W)$.

For every $z$ in $[z_0,0)$ we have $\dot{\lambda_0^z}(t_z)\neq 0$.
By continuity of the applications $z\mapsto t_z$ and $(t,z)\mapsto \dot{\lambda_0^z}(t)$, 
there  exist 
$\alpha,c_2>0$ 
such that $|\dot{\lambda_0^z}(t)|>c_2>0$ 
for every $z\in [z_0,0]$ such that $t_z\le (t_1+1)/2$ and every $t\in [t_z - \alpha,t_z + \alpha]$.
By continuity of the application $(t,z)\mapsto \lambda_0^z(t)$, moreover, 
we get the existence of $c_3>0$ 
such that $|\lambda_0^z(t)|>c_3>0$ for every 
$z\in [z_0,0]\setminus W$ and every $t\in [0,1] \setminus [t_z - \alpha,t_z + \alpha]$, also
for every $z\in [z_0,0]\cap W$ and every  $t\in [0,t_1] \setminus [t_z - \alpha,t_z + \alpha]$. 

For $z\in [z_0,0]\setminus W$ and $t\in [0,1]$, 
we write
\begin{align*}
\int_0^t e^{\frac{2i}{\epsilon}\int_0^s\lambda_0^z(r)dr}ds=&
\int_{[0,t]\cap [0,t_1-\alpha]}e^{\frac{2i}{\epsilon}\int_0^s\lambda_0^z(r)dr}ds+
\int_{[0,t]\cap [t_1-\alpha,t_1+\alpha]}e^{\frac{2i}{\epsilon}\int_0^s\lambda_0^z(r)dr}ds\\
&+
\int_{[0,t]\cap [t_1+\alpha,1]}e^{\frac{2i}{\epsilon}\int_0^s\lambda_0^z(r)dr}ds,
\end{align*}
and we conclude, up a change of time variable, 
by applying Corollary~\ref{estimationensemble} (on $[0,t]\cap [t_1-\alpha,t_1+\alpha]$, with $k=2$) 
and Lemma~\ref{k1ens} (on $[0,t]\cap [0,t_1-\alpha]$ and $[0,t]\cap [t_1+\alpha,1]$).  

We conclude similarly for $z\in W\cap [z_0,0]$, by splitting $[0,1]$ in the intervals 
$[0,\min(t_1,t_z-\alpha)]$, $[\min(t_1,t_z-\alpha),t_1]$, and $[t_1,1]$ and by applying 
Corollary~\ref{estimationensemble}, Lemma~\ref{k1ens}, and \eqref{local-estimate}. 
\end{proof}

Proposition~\ref{bound} and \ref{esti2} allow us to apply Theorem~\ref{Unifadiab} and deduce  the following ensemble adiabatic approximation result.
 \begin{teo}[Semi-conical case]\label{rouge}
 Let $(u,v)$ be a regular $C^{\infty}$ control path satisfying condition (\textbf{C}).
 Let $\psi^{z}_{\epsilon}$ be the solution of Equation~\eqref{adpersch}, where
 $\tilde\psi^{z}_{0}=\phi_{-}^{z}(u(0),v(0))$  for every $z\in (z_0,z_1]$ and  $\tilde\psi^{z_0}_{0}=\lim_{t\to 0^+} \phi_{-}^{z_0}(u(t),v(t))$. 
Set $T_{\epsilon}(z)=|\langle \psi^{z}_{\epsilon}(\frac{1}{\epsilon}),\phi_{+}^z(u(1),v(1))\rangle|$. 
Then
  $$\lim_{\epsilon\to 0}T_{\epsilon}(z)=\left\{\begin{array}{ll}
    0 & \text{if} \ z\in [0,z_1]\cup \{z_0\},\\
     1 & \text{if} \ z\in (z_0,0),
     \end{array}\right.
     $$
 the convergence being uniform w.r.t. $z\in [z_0,z_1]$.
 More precisely, we have $T_{\epsilon}(z)=O(\epsilon^{1/3})$ for $z\in [0,z_1]\cup \{z_0\}$.
\end{teo}

\subsection{The control path $(u,v)$ exits from $\pi(f)$. }
By similar arguments as those developed in \cite{Ensemble}, we get the following proposition.
\begin{prop}[Conical exit]\label{conikexit}
Let $f\in C^{\infty}(\R^3,\R^2)$.
Let $(u_1,v_1,z_1)$ be a F-conical intersection for $f$.
Let  $N$ be  a neighborhood of $(u_1,v_1,z_1)$ in $\R^3$ such that
$Z(f)\cap N$ is made of F-conical intersections only and $\pi(Z(f) \cap N)$ is a $C^{\infty}$ submanifold of $\R^2$.
Let $(u_0,v_0,z_0)\in Z(f)\cap N$ be such that $z_0<z_1$.
Consider a regular $C^3$ control path $(u(t),v(t))_{t\in [0,1]}$ and a time $t_1\in (0,1)$ such that $(u(0),v(0))=(u_0,v_0)$, $(u(t_1),v(t_1))=(u_1,v_1)$, $(u(t),v(t))\in \pi(Z(f)\cap N)$ for $t\in [0,t_1]$, and $(u(t),v(t))\notin \pi(f)$ for $t>t_1$.
For every $z\in \R$, consider $\theta_z\in C^{2}([0,1],\R)$ and $\lambda_z\in C^{2}([0,1],\R)$ as in Theorem~\ref{Unifadiab}. 
  Then $(t,z)\mapsto \dot{\theta}_z(t)$ and $(t,z)\mapsto \ddot{\theta}_z(t)$ are uniformly bounded with respect to $(t,z)\in [0,1]\times[z_0,z_1]$, and there exists $c>0$ such that for every $z\in (z_0,z_1]$ and for every $t$ in $[0,1]$,
\[\left|\int_0^t e^{\frac{2i}{\epsilon}\int_0^s\lambda_z(x)dx}ds\right|\leq c\epsilon^{\frac1{2}},\qquad \forall\epsilon>0.\]
\end{prop}
  
 \begin{cor}
  Let $f,(u,v)$ as in Proposition~\ref{conikexit}.
Define $\psi_{\epsilon}^z$ as the solution of (\ref{adpersch}) with $\tilde\psi^z_0=\phi_{-}^z(u(0),v(0))$ for $z\ne 0$ 
and $\tilde\psi^{z_0}_0=\lim_{t\to 0}\phi_{-}^{z_0}(u(t),v(t))$. 
Then for every $z\in (z_0,z_1]$ and for every $t$ in $[0,1]$,
    $$|\langle \psi_{\epsilon}^{z}({1}/{\epsilon}),\phi_{+}^z(u(1),v(1))\rangle|=1+O(\epsilon^{1/2}).$$
Moreover, for every $\bar{z}\leq z_0$, we have, uniformly w.r.t. $z\in [\bar{z},z_0]$,
$$|\langle \psi_{\epsilon}^{z}({1}/{\epsilon}),\phi_{+}^z(u(1),v(1))\rangle| = O(\epsilon^{1/2}).$$
\end{cor}

\subsection{Proof of Theorem~\ref{thm:intro2}}

\begin{proof}[Proof of Theorem~\ref{thm:intro2}]
Consider a regular $C^4$ control path
$(\eta(t))_{t\in [0,1]}=(u(t),v(t))_{t\in [0,1]}$ such that $\eta(0)=\eta(1)$, $\eta(t_0)=(u_0,v_0)$, $\eta(t_1)=(u_1,v_1)$, $\eta(t)\in \pi(\gamma)$ for $t\in [t_0,t_1]$, and $\eta(t)\notin \pi(f)$ for  $t\notin [t_0,t_1]$.
Under these hypotheses, we can define, for every $z\in \R$, $\theta_z\in C^{2}([0,1],\R)$ and $\lambda_z\in C^{2}([0,1],\R)$ along the path $\eta$, as required in Theorem~\ref{Unifadiab}.

For $t\in [t_0,t_1]$, the hypothesis of non-existence of self-intersections for $\pi(\gamma)$ guarantees that we can apply the same arguments as those used in the proof of Proposition~\ref{esti2} for the normal form in order to get 
 \[\left|\int_{t_0}^t e^{\frac{2i}{\epsilon}\int_0^s\lambda_z(r)dr}ds\right| \leq C \epsilon^{\frac{1}{3}},\qquad \forall\epsilon>0 \] where $C>0$ is independent of $(t,z)\in [t_0,t_1]\times [z_0,z_1]$.
Moreover by Proposition~\ref{bound}, 
$(t,z)\mapsto \dot{\theta}_z(t)$ and  $(t,z)\mapsto \ddot{\theta}_z(t)$ are bounded on $[t_0,t_1]\times [z_0,z_1]$.

Under the assumptions that $z\in [z_0,z_1]$ for every $z$ and $t$ such that $(u(t),v(t),z)\in \gamma$ and that $(u_1,v_1,z_1)$ is a F-conical intersection, we can apply Proposition~\ref{conikexit} and get that $(t,z)\mapsto \dot{\theta}_z(t)$ and $(t,z)\mapsto \ddot{\theta}_z(t)$ are uniformly bounded with respect to $(t,z)\in [t_1,1]\times[z_0,z_1]$, and there exists $c>0$ such that for every $z\in [z_0,z_1]$ and for every $t$ in $[t_1,1]$,
\[\left|\int_{t_1}^t e^{\frac{2i}{\epsilon}\int_0^s\lambda_z(x)dx}ds\right|\leq c\epsilon^{\frac1{2}},\qquad \forall\epsilon>0.\]
By similar arguments for $t\in [0,t_0]$, we obtain that on the whole interval $t\in [0,1]$, $(t,z)\mapsto \dot{\theta}_z(t)$ and $(t,z)\mapsto \ddot{\theta}_z(t)$ are uniformly bounded with respect to $(t,z)\in [0,1]\times[z_0,z_1]$, and by triangular inequality,  there exists $\tilde{C}>0$ such that for every $z\in [z_0,z_1]$ and for every $t$ in $[0,1]$,
\[\left|\int_{0}^t e^{\frac{2i}{\epsilon}\int_0^s\lambda_z(x)dx}ds\right|\leq \tilde{C} \epsilon^{\frac1{3}},\qquad \forall\epsilon>0.\]
We get the expected result by applying Theorem~\ref{Unifadiab}.
\end{proof}

\section{Extension to $n$-level systems \label{NLEVEL}}
The goal of this section is to extend Theorem~\ref{thm:intro2} of ensemble controllability between the eigenfunctions to the case of $n$-level systems.
\subsection{Generic assumptions on n-level Hamiltonians and adiabatic decoupling}\label{nlevel}
In this section, we show that the study of a $n$-level real Hamiltonian can be reduced locally to the study of a $2$-level Hamiltonian in the adiabatic regime and that such a transformation preserves the codimension of the generic conditions expressed in Section \ref{Gencond}. Such a reduction allows us to define a semi-conical intersection model for a $n$-level real Hamiltonian.

For every  $H\in S_n(\R)$ denote by $(\lambda_j(H))_{j=1}^n$ the spectrum of $H$, where $j\mapsto \lambda_j(H)$ is the nondecreasing sequence of eigenvalues of $H$ repeated according to their multiplicities. We write $(\phi_1(H),\dots,\phi_n(H))$ to denote an orthonormal basis of associated eigenvectors.

Next lemma is a classical result of continuity of the spectrum (see, for instance \cite{reedsimon}).
\begin{lem}\label{cont}
Let $H_0 \in S_n(\R)$ and $j\in \{1,\dots,n-1\}$ be such that $\lambda_j(H_0),\lambda_{j+1}(H_0)$ are separated from the rest of the spectrum of $H_0$.
Then, there exists a neighborhood $V$ of $H_0$ in $S_n(\R)$ 
and a Jordan curve $c$  in $\C$ separating $\left\{\lambda_q(H) \mid q\in \{j,j+1\} ,\ H\in V \right\}$  from $\cup_{H\in V}(\text{Spectrum} (H) \setminus \left\{\lambda_j(H),\lambda_{j+1}(H)\right\})$.
\end{lem}
From now on, we consider $H_0,j,c,V$ verifying 
Lemma~\ref{cont}. 
For all $H\in V$, we consider $P_{j,j+1}(H)=\frac{1}{2i\pi}\int_{c} (H-\xi)^{-1} d\xi$.
Notice that, $P_{j,j+1}(H)$ is a real matrix because $H$ is real.
By construction of $c$, $V\ni H\mapsto P_{H}$ is smooth.
Up to reducing $V$, for every $H$ we can consider an orthogonal mapping $I(H):\R^2 \to \text{Im} (P_{j,j+1}(H))$ such that $V\ni H \mapsto I(H)$ is smooth.
For every $H\in V$ define $\pi_{j,j+1}(H)=\text{Im}(P_{j,j+1}(H))$, 
$I^{-1}(H)$ as the inverse of $I(H)$ on $\pi_{j,j+1}(H)$ and 
\[F(H)=I^{-1}(H) H I(H)\in S_2(\R).\]
Notice that 
$I^{-1}(H)=\transpose I(H)$.

Consider $H \in C^{\infty}(\R^k,S_n(\R))$ such that $H(0)=H_0$, and denote by $W$ a neighborhood of $0$ in $\R^k$ such that $H(u)\in V$  for every $u\in W$.
 Define $h\in C^{\infty}(\R^k,S_2(\R))$ such that for every $u\in W$,  $h(u)=(F\circ H)(u)$. 
We say that $h$ is a \emph{reduced Hamiltonian} for $H$. Notice that if $\phi\in \C^2$ is an eigenvector of $h(u)$ associated with the eigenvalue $\lambda\in \R$ then $I(H(u))\phi$ is an eigenvector of $H(u)$ associated with the same eigenvalue $\lambda$. 
We deduce from this, as it has been already used in \cite{Ensemble}, that the regularity of the eigenpairs of $H$ with respect to $u\in W$ can be deduced from the regularity of those of $h$.

\begin{prop}
$F$ is a submersion from $V$ to $S_2(\R)$.
\begin{proof}
Consider $A\in V$.
Define $\psi_1=I(A)e_1,\psi_2=I(A)e_2$ where $(e_1,e_2)$ is the canonical basis of $\C^2$.
\begin{itemize}
\item 
Define $H=h_{11}\psi_1 \transpose \psi_1+h_{22}\psi_2 \transpose \psi_2+h_{12}\psi_1 \transpose \psi_2+h_{12}\psi_2 \transpose \psi_1$ with $h_{11},h_{22},h_{12} \in \R$.
By direct computations, we get $$\transpose I(A)HI(A)=\begin{pmatrix}
h_{11}&h_{12}\\
h_{12}&h_{22}
\end{pmatrix}.$$
Hence, the application $S_n(\R)\ni H \mapsto \transpose{I(A)}H I(A)\in S_2(\R)$ is surjective.
\item For $H\in S_n(\R)$  such that $A+H\in V$, 
\begin{align*}
F(A+H)&=\transpose{I(A+H)} (A+H) I(A+H)\\
&=F(A)+\transpose I(A)HI(A)+\transpose{DI_A}(H)A I(A)+\transpose{I(A)} A DI_{A}(H)+o(H).
\end{align*}
Hence, $\forall H\in S_n(\R), \ DF_A(H)=\transpose I(A)HI(A)+\transpose{DI_A}(H)A I(A)+\transpose{I(A)} A DI_{A}(H)$.
Let us consider $H=h_{11}\psi_1 \transpose \psi_1+h_{22}\psi_2 \transpose \psi_2+h_{12}\psi_1 \transpose \psi_2+h_{12}\psi_2 \transpose \psi_1$.
Then, we have $DF_A(H)=\transpose I(A)HI(A)$.
Hence, $F$ is a submersion.
\end{itemize}

\end{proof}
\end{prop}
Using classical facts on the composition of $k$-jets (see \cite{Lib}), we obtain the following result.
\begin{prop}\label{submersion}
Consider $\mathcal{F}:J^2(W,S_n(\R)) \to J^2(W,S_2(\R))$ such that for every $u\in W$, $j^2(h)(u)=\mathcal{F}(j^2(H)(u))$.
Then $\mathcal{F}$ is a submersion.
\end{prop}

It follows that if $S$ is a codimension $q$ smooth submanifold of  $J^2(W,S_2(\R))$, then $\mathcal{F}^{-1}(S)$ is a codimension $q$ smooth submanifold of $J^2(W,S_n(\R))$.
This can be used to deduce generic properties for $H\in C^{\infty}(W,S_n(\R))$ from generic properties for $h\in C^{\infty}(W,S_2(\R))$.

\subsection{Adiabatic decoupling}
We present here some results of adiabatic decoupling, 
adapted from \cite{Teu}.
\begin{teo}[Adiabatic decoupling]\label{adiabaticdecoupling}
Let $H\in C^{\infty}(\R^k,S_n(\R))$ and $j\in\{1,\dots,n-1\}$ be such that $\left\{\lambda_q(H(u)) \mid q\in \{j,j+1\} \right\}$ is separated from  $\text{Spectrum} (H(u)) \setminus \left\{\lambda_q(H(u)) \mid q\in \{j,j+1\} \right\}$ for $u$ in a neighborhood $W$ of $0$ in $\R^k$. 
Define, for every $u\in W$, $I(H(u))$ and $h(u)$ as in Section~\ref{nlevel}.
Consider a $C^2$ regular path $u:[0,1] \to W$ such that there exist $\ell\in \N$ and $C^{\ell}$  functions $\Lambda_j, \Lambda_{j+1}:[0,1] \to \R$ such that for every $t\in [0,1]$, $\{\Lambda_j(t),\Lambda_{j+1}(t)\}=\{\lambda_j(H(u(t))),\lambda_{j+1}(H(u(t)))\}$
and that $h$ admits $C^2$ eigenvectors along $u$.

Assume that there exists $c>0$ such that
 \begin{equation}\label{estimate:lambda}
\left|\int_0^t e^{\frac{i}{\epsilon}\int_0^s ( \Lambda_{j+1}(x)-\Lambda_{j}(x) )dx}ds\right|\leq c\epsilon^{\frac1{\ell+1}},\qquad \forall t \in [0,1].
\end{equation}
Let $\tilde{\psi}_0\in \C^2$.
Then the solutions $\psi_{\epsilon}$ and $\tilde{\psi_{\epsilon}}$ of, respectively, \[i\frac{d\psi}{dt}=H(u(\epsilon t))\psi,\; \psi(0)=I(H(u(0)))\tilde{\psi}_0,\quad \mbox{and}\quad  i\frac{d\tilde{\psi}}{dt}=h(u(\epsilon t))\tilde{\psi},\; \tilde{\psi}(0)=\tilde{\psi}_0,\]
 are such that $\psi_\epsilon(1/ \epsilon)$ is $O(\epsilon^{\frac{1}{\ell+1}})$-close to $I(H(u(1)))\tilde{\psi}_\epsilon(1/ \epsilon)$.

\begin{proof}
Define for $q\in \{1,2\}$, and for every $u\in W$, $\psi_q(u)=I(H(u))e_q$, where $(e_1,e_2)$ is the canonical basis of $\C^2$.
Define $\psi^{\text{eff}}_{\epsilon}(t)$ as the solution for $t\in [0,\frac{1}{\epsilon}]$, of
\begin{equation}
i\frac{d\psi^{\text{eff}}_{\epsilon}(t)}{dt}=\left(
h(u(\epsilon t))-i\epsilon \begin{pmatrix}
0&\langle \dot{\psi_j}(u(\epsilon t)),\psi_{j+1}(u(\epsilon t)) \rangle\\
\langle \dot{\psi}_{j+1}(u(\epsilon t)),\psi_{j}(u(\epsilon t)) \rangle&0
\end{pmatrix}\right)\psi^{\text{eff}}_{\epsilon}(t),
\end{equation}
with $\psi^{\text{eff}}_{\epsilon}(0)=\tilde{\psi}_0$.

By \cite[Theorem 1.4]{Teu}, there exists $C>0$, such that, for every $t\in [0,\frac{1}{\epsilon}]$,
\[\|\psi_{\epsilon}(t) - I(h(u(\epsilon t)))\psi^{\text{eff}}_{\epsilon}(t)\| \leq C \epsilon , \qquad \forall \epsilon>0.\]
Under the assumptions of the theorem, we can consider a $C^2$ basis of eigenvectors of $h$ along $u$.
Hence, by the same arguments as those used in  Section~\ref{adiabexplained} in  order to prove Theorem~\ref{adiabad}, 
there exists $c>0$ such that, for every $t\in [0,\frac{1}{\epsilon}]$,
\[\|\tilde{\psi_{\epsilon}}(t) - \psi^{\text{eff}}_{\epsilon}(t)\| \leq c \epsilon^{\frac{1}{\ell+1}} , \qquad \forall \epsilon>0.\]
We get the expected result by triangular inequality.
\end{proof}
\end{teo}

In the ensemble case, using estimates that are uniform with respect to the parameter $z$, we get the following extension of Theorem~\ref{adiabaticdecoupling}. 
\begin{teo}[Adiabatic decoupling for parametric systems]\label{ensemble:decoupling}
Let $H\in C^{\infty}(\R^{k+1},S_n(\R))$ and $j\in\{1,\dots,n-1\}$ be such that $\left\{\lambda_q(H(u,z)) \mid q\in \{j,j+1\} \right\}$ is separated from  $\text{Spectrum} (H(u,z)) \setminus \left\{\lambda_q(H(u,z)) \mid q\in \{j,j+1\} \right\}$ for $u$ in a neighborhood $W$ of $0$ in $\R^k$ and $z\in [z_0,z_1]$.
Define, for every $(u,z)\in W\times [z_0,z_1]$, $I(H(u,z))$ and $h(u,z)$ as in Section~\ref{nlevel}.
Consider a $C^2$ regular path $u:[0,1] \to W$ and $\tilde{\psi}_0^z\in \C^2$, for every $z\in [z_0,z_1]$.
Let $\ell\in \N$ and assume that for every $z\in [z_0,z_1]$, there exist $C^{\ell}$ functions $\Lambda_j^z, \Lambda_{j+1}^z:[0,1] \to \R$ such that for every $t\in [0,1]$, $\{\Lambda_j^z(t),\Lambda_{j+1}^z(t)\}=\{\lambda_j^z(u(t)),\lambda_{j+1}^z(u(t))\}$
and that, for every $z\in [z_0,z_1]$, $h(u(\cdot),z)$ admits $C^2$ eigenvectors $\Phi_j^z,\Phi_{j+1}^z$ such that $(t,z)\mapsto \frac{d\Phi_{q}^z(t)}{dt}$ and $(t,z)\mapsto \frac{d^2\Phi_{q}^z(t)}{dt^2}$ are bounded uniformly with respect to $(t,z)\in [0,1]\times [z_0,z_1]$, for every $q\in \{j,j+1\}$.
Assume that there exists $c>0$ such that 
 \begin{equation}\label{estimate:lambda}
\left|\int_0^t e^{\frac{i}{\epsilon}\int_0^s ( \Lambda_{j+1}^z(x)-\Lambda_{j}^z(x) )dx}ds\right|\leq c\epsilon^{\frac1{\ell+1}},\qquad \forall t \in [0,1],\qquad \forall z \in [z_0,z_1].
\end{equation}

Then the solutions $\psi^{z}_{\epsilon}$ and $\tilde{\psi}^{z}_{\epsilon}$ of, respectively, \[i\frac{d\psi^z}{dt}=H(u(\epsilon t))\psi^z,\; \psi^z(0)=I(H(u(0),z))\tilde{\psi}^z_0,\quad \mbox{and}\quad i\frac{d\tilde{\psi}^z}{dt}=h(u(\epsilon t),z)\tilde{\psi}^z,\; \tilde{\psi}^z(0)=\tilde{\psi}^z_0,\] are such that $\psi^z(1/ \epsilon)$ is $O(\epsilon^{\frac{1}{\ell+1}})$-close to $I(H(u(1),z))\tilde{\psi}^{z}(1/ \epsilon)$, uniformly w.r.t. $z\in [z_0,z_1]$.
\end{teo}

\subsection{Semi-conical intersections for $n$-level quantum systems}
Let $H\in C^{\infty}(\R^k,S_n(\R))$ and $j\in\{1,\dots,n-1\}$ be such that $\left\{\lambda_q(H(u)) \mid q\in \{j,j+1\} \right\}$ is separated from  $\text{Spectrum} (H(u)) \setminus \left\{\lambda_q(H(u)) \mid q\in \{j,j+1\} \right\}$ for $u$ in a neighborhood $W$ of $0$ in $\R^k$. 
Define, for every $u\in W$, $I(H(u))$ and $h(u)$ as in Section~\ref{nlevel}.
Define for $q\in \{1,2\}$, and for every $u\in W$, $\psi_q(u)=I(H(u))e_q$, where $(e_1,e_2)$ is the canonical basis of $\C^2$.
Then we have the identity
 $$h(u)=\begin{pmatrix}
\langle \psi_1(u) , H(u) \psi_1(u) \rangle & \langle \psi_1(u) , H(u) \psi_2(u) \rangle\\
\langle \psi_1(u) , H(u) \psi_2(u) \rangle & \langle \psi_2(u) , H(u) \psi_2(u) \rangle\\
\end{pmatrix}.$$
 \begin{defi}
 By removing the trace of $H(u)$, define the \emph{reduced zero-trace Hamiltonian} of $H$ as, for every $u\in W$,
\[h_{\rm red}(u)=\begin{pmatrix}
f_1(u)&f_2(u)\\
f_2(u)&-f_1(u)
\end{pmatrix},\]
with $f_1(u)=\frac{1}{2}\left(\langle \psi_1(u) , H(u) \psi_1(u) \rangle - \langle \psi_2(u) , H(u) \psi_2(u) \rangle\right)$ and
$f_2(u)=\langle \psi_1(u) , H(u) \psi_2(u) \rangle$.
\end{defi}

Assume $k=2$. By a slight abuse of notations, write $u:=(u,v)\in U$, where $U$ is a connected open neighborhood of the origin in $\R^2$.

Next proposition, which follows by direct computations, states that conicity properties do not depend on the choice of the unitary transformation $I(H(u,v)): \C^2 \to {\rm Im} (P_{j,j+1}(H(u,v)))$.
\begin{prop}
Let $f
\in C^{\infty}(U,\R^2)$ and $R\in C^{\infty}(U,O_2(\R))$. Define $\tilde{f}\in C^{\infty}(U,\R^2)$ such that $H_{\tilde{f}}(u,v)=R(u,v)H_f(u,v)\transpose R(u,v)$ for every $(u,v)$ in $U$.
Then 
\begin{itemize}
\item  $0$ is \emph{conical} for $f$ if and only if $0$ is \emph{conical} for ${\tilde{f}}$;
\item $0$ is \emph{semi-conical} for $f$ if and only if $0$ is \emph{semi-conical} for ${\tilde{f}}$. Moreover, their non-conical directions are the same.
\end{itemize}
\end{prop}

Similarly, one can show that in the ensemble case, $F$-conical intersections and $F$-semi-conical intersections are invariant under such a orthogonal mapping, possibly depending on the parameter $z$.
We define semi-conical intersections of eigenvalues for $H(\cdot)\in C^{\infty}(U,S_n(\R))$ and F-conical (respectively F-semi-conical intersections) for $H(\cdot)\in C^{\infty}(U\times \R,S_n(\R))$ (see Section~\ref{normal} for precise definitions of these notions for two level systems)  as follows.
\begin{defi}\label{sc:nlevel}
Let $j\in \{1,\dots,n-1\}$.
\begin{itemize}
\item We say that \emph{$(\bar{u},\bar{v})\in U$ is a semi-conical intersection for $H\in C^{\infty}(U,S_n(\R))$} between the levels $j$ and $j+1$ if and only if there exists a unitary mapping $I(H(u,v)): \C^2 \to {\rm Im} (P_{j,j+1}(H(u,v)))$, $C^{\infty}$ with respect to $(u,v)\in U$, such that $(\bar{u},\bar{v})$ is a semi-conical intersection for the associated reduced Hamiltonian  $h_{\rm red}\in C^{\infty}(U,S_2(\R))$.
\item We say that \emph{$(\bar{u},\bar{v},\bar z)\in U\times \R$ is a F-conical (respectively F-semi-conical) intersection for $H\in C^{\infty}(U\times \R,S_n(\R))$} between the levels $j$ and $j+1$ if and only if there exists a unitary mapping $I(H(u,v,z)): \C^2 \to {\rm Im}(P_{j,j+1}(H(u,v,z)))$, $C^{\infty}$ with respect to $(u,v,z)\in U\times \R$, such that $(\bar{u},\bar{v},\bar z)$ is a F-conical (respectively F-semi-conical) intersection for the associated reduced Hamiltonian  $h_{\rm red}\in C^{\infty}(U\times\R,S_2(\R))$.
\end{itemize}
\end{defi}

By Proposition~\ref{submersion}, we get that $F$-conical intersections and $F$-semi-conical intersections as defined in Definition~\ref{sc:nlevel} are generic for $H\in C^{\infty}(\R^3,S_n(\R))$ endowed with the Whitney topology.

\begin{oss}
For $j\in \{1,\dots, n-1\}$ set $Z_{j}=\{(u,v,z)\in U\times \R \mid \lambda_j(u,v,z)=\lambda_{j+1}(u,v,z)\}$. By Definition~\ref{sc:nlevel}, we have the expected result (see Proposition~\ref{order:tangency} for the same property for two-level systems)  that if $(\bar{u},\bar{v},\bar{z})$ is a F-semi-conical intersection between the levels $j$ and $j+1$, then $Z_{j}$ is tangent to the plane $z=\bar{z}$ at the point $(\bar{u},\bar{v},\bar{z})$ and, considering a local smooth and regular parametrization $(u(t),v(t),z(t))_{t\in [0,1]}$ of $Z_j$ and $\bar{t}\in [0,1]$ such that $(u(\bar{t}),v(\bar{t}),z(\bar{t}))=(\bar{u},\bar{v},\bar{z})$, we have $\dot{z}(\bar{t})=0$ and $\ddot{z}(\bar{t})\neq 0$.
\end{oss}

\subsection{Controllability result}

We consider the controlled Schr\"{o}dinger equation in $\C^n$, $n\in \N$,
\begin{equation}\label{ensemblecontrollability}
i\frac{d\psi(t)}{dt}=H(u(t),v(t),z)\psi(t).
\end{equation} 

\begin{defi}
Let $z_0,z_1\in \R$.
We say that system~(\ref{ensemblecontrollability}) is \emph{ensemble approximately controllable between eigenstates} if for every $\epsilon>0$, $j,k\in \{1,\dots,n\}$, and $(u_0,v_0),(u_1,v_1)\in \text{U}$ such that $\lambda_j(u_0,v_0,z)$ and $\lambda_k(u_0,v_0,z)$  are simple for every $z\in [z_0,z_1]$, there exists a control $(u(\cdot),v(\cdot))\in L^\infty([0,T], \text{U})$ such that for every $z\in [z_0,z_1]$ the solution of (\ref{ensemblecontrollability}) with initial condition $\psi^z(0)=\phi_{j}^z(u_0,v_0)$ satisfies $\|\psi^z(T)-e^{i\theta}\phi_{k}^{z}(u_1,v_1)\|<\epsilon$ for some $\theta \in \R$ (possibly depending on $z$ and $\epsilon$).
\end{defi}

For $j\in \{1,\dots, n-1\}$, let us denote by $\gamma_j$ the set $\{(u,v,z)\in \text{U}\times [z_0,z_1] \mid \lambda_j(u,v,z)= \lambda_{j+1}(u,v,z) \}$. Let, moreover, $\gamma_0=\gamma_n=\emptyset$.
Denote by $\pi$ the projection $\pi:(u,v,z)\mapsto (u,v)$.

\smallskip

\noindent{\bf Assumption} $A_j$. {\it There exists a connected component $\hat{\gamma}_j$ of $\gamma_j$ such that 
\begin{itemize}
\item $\hat{\gamma}_j$ is a one-dimensional submanifold of $\R^3$ made of F-conical intersections and F-semi-conical intersections only;
\item There exist $(u_0,v_0)\in \text{U}$ and $(u_1,v_1)\in \text{U}$ such that $(u_0,v_0,z_0), (u_1,v_1,z_1)\in \hat{\gamma}_j$ are F-conical intersections for $H$ ;
\item  $\pi(\hat{\gamma}_j)$ is a $C^{\infty}$ embedded curve of $\R^2$ without self-intersections, which is contained in $\text{U} \setminus (\pi(\gamma_{j-1})\cup \pi(\gamma_{j+1}))$.
\end{itemize}
}
\smallskip

Using the control strategy proposed in Theorem~\ref{thm:intro2} and the result of adiabatic decoupling proposed in Theorem~\ref{ensemble:decoupling}, we get the following result.

\begin{teo}
Consider a $C^\infty$ map $\text{U}\times [z_0,z_1]  \ni (u,v,z) \mapsto H(u,v,z)\in S_n(\R)$.
Let assumption $A_j$ be satisfied for every $j\in \{1,\dots, n-1\}$.
Then system~(\ref{ensemblecontrollability}) is ensemble approximately controllable between eigenstates.
\end{teo}

\appendix

\section{Averaging theorems and estimates of oscillatory integrals}
The following theorem is a quantitative version in $u(n)$ of a more general averaging result stated in \cite[Lemma 8.2]{Agra}. Its proof is similar to the proof of \cite[Lemma 8.2]{Agra} using an explicit inequality that yields the speed of convergence of order $O(\epsilon)$.
\begin{teo}\label{av}
Consider $A$ and $(A_{\epsilon})_{\epsilon>0}$ in $C^\infty([0,1],u(n))$ and assume that 
$A_{\epsilon}(\tau)$ is uniformly bounded w.r.t. $(\tau,\epsilon)$.
Denote the flow of the equation $\frac{dx(\tau)}{d\tau}=A(\tau)x(\tau)$ at time $\tau$  by $P_{\tau}\in U(n)$ and the flow of the equation $\frac{dx(\tau)}{d\tau}=A_{\epsilon}(\tau)x(\tau)$ at time $\tau$  by $P_{\tau}^{\epsilon}\in U(n)$.
If \[\int_0^\tau A_{\epsilon}(s)ds=\int_0^\tau A(s) ds +O(\epsilon)\]
then \[P_{\tau}^{\epsilon}=P_{\tau}+O(\epsilon),\]
both estimates being uniform w.r.t. $\tau\in  [0,1]$.
\end{teo}

\begin{oss}
Notice that in \cite[Lemma 8.2]{Agra}, the hypothesis that $\|A_\epsilon\|_\infty$ is bounded w.r.t. $\epsilon$ is not explicitly mentioned, but it is necessary for concluding that $P_{\tau}^{\epsilon}\to P_{\tau}$ as $\epsilon\to 0$.
\end{oss}

A direct consequence of Theorem~\ref{av} is the following.
\begin{cor}[Quantum two-level systems averaging]\label{avesti}
Let $v,\varphi : [0,1] \to \R$ be two smooth functions and, 
for every $\epsilon>0$, denote by $P_{\tau}^{\epsilon}$ the flow at time $\tau\in [0,1]$ of 
\begin{equation*} 
i\frac{dX}{d\tau}=\begin{pmatrix}
0&v(\tau)e^{\frac{i}{\epsilon}\varphi(\tau)}\\
v(\tau)e^{-\frac{i}{\epsilon}\varphi(\tau)}&0
\end{pmatrix}X(\tau).
\end{equation*}
If $|\int_0^\tau v(s)e^{\frac{i}{\epsilon}\varphi(s)}ds| \leq c \epsilon^q$, where $q$ is a positive real number and $c>0$ is independent of $\epsilon,\tau$, then $P_{\tau}^{\epsilon}$ satisfies
$P_{\tau}^{\epsilon}=\mathrm{Id}+O(\epsilon^q)$.
\end{cor}
We recall a classical result (see \cite{Stein}) which is useful to estimate oscillatory integrals.
\begin{lem}[Van Der Corput]\label{vdc}
Let $k\in \N$ and 
$\varphi:[a,b]\to \R$ be smooth and such that $|\varphi^{(k)}(x)| \geq 1$ for all $x\in [a,b]$.
Assume  either that $k\geq 2$ or that $k=1$ and $\varphi'$ is monotone.
Then 
$$\left| \int_a^b e^{i\varphi(x)/ \epsilon}dx \right| \leq c_k \epsilon^{1/k},$$ 
where $c_k$ is independent of $\varphi$ and $\epsilon$. 
\end{lem}

In the case $k=1$, if $\varphi'$ is not monotone we may lose the uniformity of the 
estimate with respect to the phase $\varphi$. 
However, we can recover by a direct integration by parts the following estimate.

\begin{lem}[The case $k=1$] \label{k1}
Let $\varphi:[a,b]\to \R$ be smooth and such that $|\varphi'(x)| \geq 1$ for all $x\in [a,b]$.
Then $$\left| \int_a^b e^{i\varphi(x)/ \epsilon}dx \right| \leq 2\epsilon+\epsilon \int_a^b \left|\frac{d}{dx} \frac{1}{\varphi'(x)}\right|dx.$$
\end{lem}

By integration by parts we also get the following results.

\begin{cor}\label{estimation}
Let $\varphi$ and $k$ be as in Lemma \ref{vdc}. Let, moreover, $v:[a,b]\to \R$ be smooth. Then
$$\left| \int_a^b v(x) e^{i\varphi(x)/ \epsilon}dx \right| \leq c_k \epsilon^{1/k}\left[|v(b)|+\int_a^b |v'(x)|dx \right]$$
where $c_k$ is the constant obtained in Lemma \ref{vdc}.
\end{cor}

\begin{cor}\label{estimation}
Let $\varphi$ be as in Lemma \ref{k1} and consider a smooth function $v:[a,b]\to \R$. 
Then $$\left| \int_a^b v(x) e^{i\varphi(x)/ \epsilon}dx \right| \leq c \epsilon$$
where $c$ is independent of $\epsilon$. 
\end{cor}

\begin{cor}\label{estimationensemble}
Consider an open subset ${\cal Z}$ of $\R$.
Assume that $\varphi:[a,b]\times {\cal Z} \to \R$ and $v:[a,b]\times {\cal Z} \to \R$ are real-valued and 
smooth with respect to the first variable $x\in [a,b]$. 
Assume that there exists $k>1$ such that $|\frac{\partial^k \varphi}{\partial x^k} (x,y)| \geq 1$ for all $x\in [a,b]$ and $y\in {\cal Z}$.
Then $$\left| \int_a^b v(x,y) e^{i\varphi(x,y)/ \epsilon}dx \right| \leq c_k \epsilon^{1/k}\left[|v(b,y)|+\int_a^b |\frac{\partial v}{\partial x}(x,y)|dx \right]$$
where $c_k$ is the constant obtained in the Lemma \ref{vdc}  (independent of $\varphi$, $y$ and $\epsilon$). 
If we assume that $v$ and $\frac{\partial v}{\partial x}$ are uniformly bounded on $[a,b]\times {\cal Z}$, then  $$\left| \int_a^b v(x,y) e^{i\varphi(x,y)/ \epsilon}dx \right| \leq d_k \epsilon^{1/k}$$ where $d_k$ depends on $v$ and is independent  of $\varphi$, $y\in {\cal Z}$ and $\epsilon$.
\end{cor} 
Next lemma is a direct consequence of Lemma \ref{k1}.

\begin{lem}\label{k1ens}
Consider a compact subset ${\cal Z}$ of $\R$.
Consider two real-valued and smooth functions $\varphi:[a,b]\times {\cal Z} \to \R$ and $v:[a,b]\times {\cal Z} \to \R$.
Assume that $|\frac{\partial \varphi}{\partial x} (x,y)| \geq 1$ for all $x\in [a,b]$ and $y\in {\cal Z}$.
Then  $$\left| \int_a^b v(x,y) e^{i\varphi(x,y)/ \epsilon}dx \right| \leq d \epsilon$$ where $d$ depends on $v$ and $\varphi$ and is independent of  $y\in {\cal Z}$ and $\epsilon$.
\end{lem}

\section{Two useful lemmas }
We recall some classical results that are derived from \cite[\textsection{9}
]{ABB19}. 

\begin{lem}\label{mal1}
Let $n\in \N$ and let $\R^n\times \R \ni (x,y)\mapsto F(x,y)\in \R$ 
be a smooth function
 vanishing on the graph $y=\eta(x)$,
 where $\eta:\R^n \to \R$ is a smooth function.
Then for every 
point $x_0\in \R^n$ there exist a neighborhood $W$ of $(x_0,\eta(x_0))$ and
a smooth function $\varphi:W\to \R$ such that
$$\forall (x,y) \in W,\qquad F(x,y)=(y-\eta(x))\varphi(x,y).$$
\end{lem}

\begin{lem}\label{coromal}
 Let $n\in \N$ and let 
$F:\begin{aligned}
&\R^n\times \R \to \R\\
&(x,y)\mapsto F(x,y)
\end{aligned}$ 
be a smooth function such that $\frac{\partial F}{\partial y} $ is vanishing on the smooth hypersurface $y=\eta(x)$. 
Then for every point $x_0\in \R^n$ there exist a neighborhood $W$ of $(x_0,\eta(x_0))$ that can be written as $W=W_1\times W_2$ where $W_1$ is an open subset of $\R^n$ and $W_2$ is an open subset of $\R$,
and smooth functions $\varphi:W \to \R$ and $f_{0}:W_1\to \R$ such that
$$\forall (x,y) \in U, F(x,y)=(y-\eta(x))^{2}\varphi(x,y)+f_{0}(x).$$
\end{lem}

\bigskip
\noindent{\bf Acknowledgements:} 
This work was supported by the ANR project SRGI ANR-15- CE40-0018 and by the ANR project Quaco ANR-17-CE40-0007-01.

\bibliographystyle{abbrv}
\bibliography{biblibonne}

\end{document}